\documentclass[a4paper,10pt]{article}

\usepackage{amsfonts}
\usepackage{amsthm}
\usepackage{amsmath}
\usepackage{amssymb}
\usepackage{mathtools} %for prescript

\usepackage[top=1in, bottom=1in, left=1in, right=1in]{geometry}
\usepackage{xcolor}

\usepackage{graphicx}
\usepackage{epstopdf}% To incorporate .eps illustrations using PDFLaTeX, etc.
\usepackage{subfigure}% Support for small, `sub' figures and tables

\usepackage{authblk}

\theoremstyle{plain}% Theorem-like structures
\newtheorem{thm}{Theorem}[section]
\newtheorem{cor}{Corollary}[section]
\newtheorem{lem}{Lemma}[section]
\newtheorem{prop}{Proposition}[section]

\theoremstyle{definition}
\newtheorem{defn}{Definition}[section]

\theoremstyle{remark}
\newtheorem{rem}{Remark}[section]
\newtheorem{ex}{Example}[section]

\usepackage{xcolor}
\usepackage{soul}
\definecolor{afcol}{rgb}{1,0,0}

\providecommand{\keywords}[1]{\textbf{\textit{Keywords:}} #1}

\begin{document}

\title{Weighted fractional calculus: a general class of operators}

\date{}

\author[1]{Arran Fernandez\thanks{Email: \texttt{arran.fernandez@emu.edu.tr}}}
\author[1]{Hafiz Muhammad Fahad\thanks{Email: \texttt{hafizmuhammadfahad13@gmail.com}}}

\affil[1]{{\small Department of Mathematics, Faculty of Arts and Sciences, Eastern Mediterranean University, Famagusta, Northern Cyprus, via Mersin 10, Turkey}}

\maketitle

\begin{abstract}
The operators of fractional calculus come in many different types, which can be categorised into general classes according to their nature and properties. We conduct a formal study of the class known as weighted fractional calculus and its extension to the larger class known as weighted fractional calculus with respect to functions. These classes contain tempered, Hadamard-type, and Erd\'elyi--Kober operators as special cases, and in general they can be related to the classical Riemann--Liouville fractional calculus via conjugation relations. Considering the corresponding modifications of the Laplace transform and convolution operations enables differential equations to be solved in the setting of these general classes of operators.
\end{abstract}


\section{Introduction}

The basic operators of calculus, namely differentiation and integration, can be generalised by allowing their order (number of repetitions) to roam outside of $\mathbb{Z}$ to more general domains such as $\mathbb{R}$ and $\mathbb{C}$. The study of such generalised operators is called fractional calculus, and this field itself has various possible levels of generality. The classical fractional derivatives and integrals, defined according to the Riemann--Liouville model, are no longer the only definitions in the literature; in the 21st century, so many new definitions have been proposed that it has become necessary to categorise them into general classes of operators for mathematical study.

Several textbooks from the late 20th century \cite{miller-ross,oldham-spanier,samko-kilbas-marichev} give the detailed theory of Riemann--Liouville and some other classical definitions of fractional integrals and derivatives. The more recent definitions are too numerous to list here, but we refer to \cite{hilfer-luchko,teodoro-machado-oliveira} for some surveys of examples, and to \cite{baleanu-fernandez} for a description of the philosophy of classification of operators in fractional calculus.

One example of a general class of operators is given by fractional calculus with analytic kernel functions \cite{fernandez-ozarslan-baleanu}, defined in 2019, which can be related to the classical Riemann--Liouville operators via infinite series. Another example is given by fractional calculus with respect to functions, first theorised by Osler in 1970 \cite{osler} and later studied in more detail in the textbooks \cite[\S18.2]{samko-kilbas-marichev} and \cite[\S2.5]{kilbas-srivastava-trujillo}. The operators in this class can be related to the classical Riemann--Liouville operators via an algebraic conjugation relation, which we describe in more detail below. The two general classes just described can also be merged into a larger class, fractional calculus with analytic kernels with respect to functions \cite{oumarou-fahad-djida-fernandez}, which is large enough to contain as special cases all of the operators in both of the above classes, as well as others such as the Hadamard-type fractional calculus which falls into neither of the above classes.

Another general class of operators, studied by Agrawal \cite{agrawal1,agrawal2} in 2012, is known as weighted (or scaled) fractional calculus with respect to functions, and its applications in variational calculus \cite{agrawal2} and probabilistic modelling \cite{kolokoltsov} have been thoroughly explored. This can also be seen as a merger of two classes, namely fractional calculus with respect to functions (as mentioned above) and weighted fractional calculus (to be defined below). A few recent papers have also studied mathematical properties of these weighted operators \cite{alrefai,jarad-abdeljawad-shah} and of the associated differential equations \cite{abdo-abdeljawad-ali-shah-jarad,bayrak-demir-ozbilge,liu-yang-feng-geng}.

Noticing this new pattern emerging in the literature, the authors were inspired to conduct a detailed study of the so-called weighted fractional calculus, in the spirit of Agrawal \cite{agrawal1,agrawal2} and Kolokoltsov \cite{kolokoltsov} who appreciated its operational properties, in order to understand its mathematical structure and provide results to make future studies much easier for anyone using these operators. In particular, by observing that a conjugation formula (similar to the one for fractional calculus with respect to functions) relates the operators of weighted fractional calculus with those of Riemann--Liouville fractional calculus, we can immediately deduce many results in the weighted theory from classical results on Riemann--Liouville.

The results and discussion of the current work will be structured as follows. Section \ref{Sec:prelim} consists of preliminary definitions and results concerning the classical fractional calculus and fractional calculus with respect to functions. Section \ref{Sec:WFC} consists of a detailed study of weighted fractional calculus, divided into subsections devoted to conjugation relations and their consequences, examples including the tempered and Kober--Erd\'elyi fractional calculi, and weighted versions of Laplace transforms and convolutions. Section \ref{Sec:WFCwrtf} extends the results of the previous section to the larger class of operators given by weighted fractional calculus with respect to functions; this larger class can again be characterised by conjugation relations, and it includes examples such as the Hadamard-type and Erd\'elyi--Kober fractional calculi. In Section \ref{Sec:concl} we conclude with a summary and some pointers for future work on weighted fractional calculus.

\section{Preliminaries} \label{Sec:prelim}

We firstly provide the definitions of the classical fractional integrals and derivatives of Riemann--Liouville and Caputo, then state some basic composition results and examples.

\begin{defn}[\cite{miller-ross,oldham-spanier,samko-kilbas-marichev,diethelm}] \label{Def:RL&C}
The Riemann--Liouville (RL) fractional integral of a function $f\in L^1(a,b)$, to order $\alpha$ in $\mathbb{R}$ or $\mathbb{C}$, is defined as
\begin{equation}
\label{RL:int}
\prescript{RL}{a}I^{\alpha}_xf(x)=\frac{1}{\Gamma(\alpha)}\int_a^x(x-t)^{\alpha-1}f(t)\,\mathrm{d}t,\qquad x\in(a,b),
\end{equation}
where we require $\mathrm{Re}(\alpha)>0$, or simply $\alpha>0$ if we assume real order.

Associated with this integral operator are two possible fractional derivative operators, named respectively after Riemann--Liouville (RL) and Caputo:
\begin{align}
\prescript{RL}{a}D^{\alpha}_xf(x)&=\frac{\mathrm{d}^n}{\mathrm{d}x^n}\Big(\prescript{RL}{a}I^{n-\alpha}_xf(x)\Big),\qquad x\in(a,b), \label{RL:der} \\
\prescript{C}{a}D^{\alpha}_xf(x)&=\prescript{RL}{a}I^{n-\alpha}_x\left(\frac{\mathrm{d}^n}{\mathrm{d}x^n}f(x)\right),\qquad x\in(a,b), \label{CAP:der}
\end{align}
where this time $\mathrm{Re}(\alpha)\geq0$, or simply $\alpha>0$ if we assume real order, and $f\in AC^n(a,b)$ where $n:=\lfloor\mathrm{Re}(\alpha)\rfloor+1$ so that $n-1\leq\mathrm{Re}(\alpha)<n$.

It is important to be aware that the Riemann--Liouville derivative \eqref{RL:der} is the analytic continuation of the Riemann--Liouville integral \eqref{RL:int} in the complex variable $\alpha$, under the convention that integrals of negative order are derivatives of positive order:
\begin{equation}
\label{RL:analcont}
\prescript{RL}{a}D^{\alpha}_xf(x)=\prescript{RL}{a}I^{-\alpha}_xf(x),\qquad\mathrm{Re}(\alpha)\geq0.
\end{equation}
This fact allows us to define both $\prescript{RL}{a}I^{\alpha}_xf(x)$ and $\prescript{RL}{a}D^{\alpha}_xf(x)$ for all values of $\alpha\in\mathbb{C}$, and apply principles of analytic continuation to extend many results from fractional integrals to fractional derivatives. In general, we may use the term ``differintegral'' to cover fractional integral and derivative operators together.
\end{defn}

\begin{lem}[\cite{oldham-spanier,samko-kilbas-marichev}] \label{Lem:RLsemigroup}
Riemann--Liouville differintegrals have semigroup properties as follows:
\begin{align*}
\prescript{RL}{a}I^{\alpha}_x\prescript{RL}{a}I^{\beta}_xf(x)&=\prescript{RL}{a}I^{\alpha+\beta}_xf(x),\qquad\alpha\in\mathbb{C},\mathrm{Re}(\beta)>0; \\
\frac{\mathrm{d}^n}{\mathrm{d}x^n}\prescript{RL}{a}D^{\alpha}_xf(x)&=\prescript{RL}{a}D^{\alpha+n}_xf(x),\qquad\alpha\in\mathbb{C},n\in\mathbb{N},
\end{align*}
where in both cases $f$ is any function such that the relevant expressions are well-defined.

Note that the operators labelled by $\alpha$ in both of these relations may be either fractional integrals or fractional derivatives, while the one labelled by $\beta$ must be a fractional integral.
\end{lem}

\begin{lem}[\cite{samko-kilbas-marichev,diethelm}]
\label{Lem:RLcompos}
The following composition properties are valid for Riemann--Liouville and Caputo differintegrals in cases where semigroup properties are not:
\begin{align*}
\prescript{RL}{a}I^{\alpha}_x\prescript{RL}{a}D^{\alpha}_xf(x)&=f(x)-\sum_{k=1}^n\frac{(x-a)^{\alpha-k}}{\Gamma(\alpha-k+1)}\cdot\lim_{x\rightarrow a^+}\prescript{RL}{a}D^{\alpha-k}_xf(x); \\
\prescript{RL}{a}I^{\alpha}_x\prescript{C}{a}D^{\alpha}_xf(x)&=f(x)-\sum_{k=0}^{n-1}\frac{(x-a)^k}{k!}\cdot\lim_{x\rightarrow a^+}\frac{\mathrm{d}^k}{\mathrm{d}x^k}f(x),
\end{align*}
where in both cases $\alpha\in\mathbb{C}$ with $\mathrm{Re}(\alpha)>0$ and $n=\lfloor\mathrm{Re}(\alpha)\rfloor+1$ while $f$ is any function such that the relevant expressions are well-defined. We also have the following relationship between the Riemann--Liouville and Caputo derivatives:
\begin{align*}
\prescript{C}{a}D^{\alpha}_xf(x)&=\prescript{RL}{a}D^{\alpha}_xf(x)-\sum_{k=0}^{n-1}\frac{(x-a)^{k-\alpha}}{\Gamma(k-\alpha+1)}\cdot\lim_{x\rightarrow a^+}\frac{\mathrm{d}^k}{\mathrm{d}x^k}f(x) \\
&=\prescript{RL}{a}D^{\alpha}_x\left(f(x)-\sum_{k=0}^{n-1}\frac{(x-a)^k}{k!}\cdot\lim_{x\rightarrow a^+}\frac{\mathrm{d}^k}{\mathrm{d}x^k}f(x)\right),
\end{align*}
where $f\in AC^n(a,b)$ and $\alpha,n$ are as before.
\end{lem}

\begin{lem}[\cite{miller-ross,oldham-spanier,diethelm}] \label{Lem:RLCspecfunc}
The Riemann--Liouville and Caputo differintegrals of certain functions are given as follows:
\begin{alignat*}{2}
\prescript{RL}{a}D^{\alpha}_x\Big((x-a)^{\beta}\Big)&=\frac{\Gamma(\beta+1)}{\Gamma(\beta-\alpha+1)}(x-a)^{\beta-\alpha},&&\qquad\alpha\in\mathbb{C},\mathrm{Re}(\beta)>-1; \\
\prescript{C}{a}D^{\alpha}_x\Big(E_{\alpha}\big(\omega(x-a)^{\alpha}\big)\Big)&=\omega E_{\alpha}\big(\omega(x-a)^{\alpha}\big),&&\qquad\omega\in\mathbb{C},\mathrm{Re}(\alpha)>0,
\end{alignat*}
where $E_{\alpha}$ is the Mittag-Leffler function. Note that the operator in the first identity can be either a fractional integral or a fractional derivative, according to the sign of $\mathrm{Re}(\alpha)$.
\end{lem}

Following the above basic introduction to RL and Caputo fractional calculus, we continue with a similar introduction to fractional calculus with respect to a function. Let $\phi:[a,b]\to\mathbb{R}$ be a strictly increasing $C^1$ function, so that $\phi'>0$ everywhere.

\begin{defn}[\cite{osler,samko-kilbas-marichev,kilbas-srivastava-trujillo}] \label{Def:WRTF}
The Riemann--Liouville fractional integral of a function $f\in L^1(a,b)$ with respect to $\phi$, to order $\alpha$ in $\mathbb{R}$ or $\mathbb{C}$, is defined as
\begin{equation}
\label{WRTF:RLint}
\prescript{RL}{a}I^{\alpha}_{\phi(x)}f(x)=\frac{1}{\Gamma(\alpha)}\int_a^x\big(\phi(x)-\phi(t)\big)^{\alpha-1}f(t)\phi'(t)\,\mathrm{d}t,\qquad x\in(a,b),
\end{equation}
where $\mathrm{Re}(\alpha)>0$, or simply $\alpha>0$ if we assume real order.

As before, associated with this integral operator are two possible fractional derivative operators, respectively the Riemann--Liouville derivative with respect to $\phi$ and the Caputo derivative with respect to $\phi$, and this time assuming that $\phi$ is a $C^{\infty}$ function:
\begin{align}
\prescript{RL}{a}D^{\alpha}_{\phi(x)}f(x)&=\left(\frac{1}{\phi'(x)}\cdot\frac{\mathrm{d}}{\mathrm{d}x}\right)^n\Big(\prescript{RL}{a}I^{n-\alpha}_{\phi(x)}f(x)\Big),\qquad x\in(a,b), \label{WRTF:RLder} \\
\prescript{C}{a}D^{\alpha}_{\phi(x)}f(x)&=\prescript{RL}{a}I^{n-\alpha}_{\phi(x)}\left(\frac{1}{\phi'(x)}\cdot\frac{\mathrm{d}}{\mathrm{d}x}\right)^nf(x),\qquad x\in(a,b), \label{WRTF:Cder}
\end{align}
where $\mathrm{Re}(\alpha)\geq0$, or simply $\alpha>0$ if we assume real order, and $f\in AC^n(a,b)$ where $n:=\lfloor\mathrm{Re}(\alpha)\rfloor+1$ so that $n-1\leq\mathrm{Re}(\alpha)<n$.

Again, the Riemann--Liouville derivative \eqref{WRTF:RLder} is the analytic continuation of the Riemann--Liouville integral \eqref{WRTF:RLint} in the complex variable $\alpha$, under the convention that integrals of negative order are derivatives of positive order. So we can define both $\prescript{RL}{a}I^{\alpha}_{\phi(x)}f(x)$ and $\prescript{RL}{a}D^{\alpha}_{\phi(x)}f(x)$ for all values of $\alpha\in\mathbb{C}$, as before.
\end{defn}

\begin{lem}[\cite{samko-kilbas-marichev,kilbas-srivastava-trujillo}] \label{Lem:WRTFconjug}
Let $\phi\in C^{\infty}[a,b]$ be a function as above, and let $Q_{\phi}$ be the functional operator of right-composition with $\phi$: $Q_{\phi}g=g\circ\phi$ for any function $g$ defined on the interval $[\phi(a),\phi(b)]$. The operators of fractional calculus with respect to functions can be expressed as algebraic conjugations, via the operator $Q_{\phi}$, of the classical fractional calculus operators:
\begin{equation}
\label{FwrtF:conjug}
\prescript{RL}{a}I^{\alpha}_{\phi(x)}=Q_\phi\circ\prescript{RL}{\phi(a)}I^{\alpha}_x\circ Q_\phi^{-1},\quad\quad\prescript{RL}{a}D^{\alpha}_{\phi(x)}=Q_\phi\circ\prescript{RL}{\phi(a)}D^{\alpha}_x\circ Q_\phi^{-1},\quad\quad\prescript{C}{a}D^{\alpha}_{\phi(x)}=Q_\phi\circ\prescript{C}{\phi(a)}D^{\alpha}_x\circ Q_\phi^{-1}.
\end{equation}
\end{lem}

Many results about fractional calculus with respect to functions can be obtained as immediate consequences of Lemma \ref{Lem:WRTFconjug} together with corresponding results in classical fractional calculus. For example, the semigroup properties of Lemma \ref{Lem:RLsemigroup} immediately give rise to corresponding semigroup properties for Riemann--Liouville differintegrals with respect to functions, simply by conjugation of the algebraic composition relations. Fractional differential equations with respect to functions can be transformed into classical fractional differential equations by simple substitutions equivalent to the conjugation relation \cite{zaky-hendy-suragan}. And the examples of Lemma \ref{Lem:RLCspecfunc} immediately give rise to the following examples of Riemann--Liouville and Caputo differintegrals with respect to functions:
\begin{align*}
\prescript{RL}{a}D^{\alpha}_{\phi(x)}\Big(\big(\phi(x)-\phi(a)\big)^{\beta}\Big)&=\frac{\Gamma(\beta+1)}{\Gamma(\beta-\alpha+1)}\big(\phi(x)-\phi(a)\big)^{\beta-\alpha},\qquad\alpha\in\mathbb{C},\mathrm{Re}(\beta)>-1; \\
\prescript{C}{a}D^{\alpha}_{\phi(x)}\Big(E_{\alpha}\big(\omega\big(\phi(x)-\phi(a)\big)^{\alpha}\big)\Big)&=\omega E_{\alpha}\big(\omega\big(\phi(x)-\phi(a)\big)^{\alpha}\big),\qquad\omega\in\mathbb{C},\mathrm{Re}(\alpha)>0.
\end{align*}

We conclude this preliminaries section by showing the definition of fractional integrals with general analytic kernels.

\begin{defn}[\cite{fernandez-ozarslan-baleanu}]
For $\alpha,\beta\in\mathbb{C}$ with positive real parts, let $A$ be a kernel function given by a power series about $0$ with radius of convergence at least $(b-a)^{\mathrm{Re}(\beta)}$. The fractional integral with kernel $A$ and parameters $\alpha,\beta$ of a function $f\in L^1(a,b)$ is defined as
\[
\prescript{A}{a}I^{\alpha,\beta}_xf(x)=\int_a^x(x-t)^{\alpha-1}A\big((x-t)^{\beta}\big)f(t)\,\mathrm{d}t,\qquad x\in(a,b).
\]
Fractional derivatives can also be defined in this general class, both of Riemann--Liouville type and of Caputo type, but we refer to \cite{fernandez-ozarslan-baleanu} for the details of these.
\end{defn}

\section{Weighted fractional calculus} \label{Sec:WFC}

\begin{defn}
	\label{Def:wRL&wC}
	The weighted Riemann--Liouville fractional integral of a given function $f\in L^1(a,b)$, with a weight function $w\in L^{\infty}(a,b)$ and order $\alpha$ in $\mathbb{R}$ or $\mathbb{C}$, is defined by
	\[
	\prescript{RL}{a}I^{\alpha}_{x;w(x)}f(x)=\frac{1}{\Gamma(\alpha)w(x)}\int_a^x(x-t)^{\alpha-1}w(t)f(t)\,\mathrm{d}t,\qquad x\in(a,b),
	\]
	where we require $\mathrm{Re}(\alpha)>0$, or simply $\alpha>0$ if we assume real order.
	
	The weighted Riemann--Liouville fractional derivative of a given function $f\in AC^n(a,b)$, with a weight function $w\in AC^n[a,b]$ and order $\alpha$ in $\mathbb{R}$ or $\mathbb{C}$, is defined by
	\[
	\prescript{RL}{a}D^{\alpha}_{x;w(x)}f(x)=\left(\frac{\mathrm{d}}{\mathrm{d}x} + \frac{w^{\prime}(x)}{w(x)}\right)^n\prescript{RL}{a}I^{n-\alpha}_{x;w(x)}f(x),\qquad x\in(a,b),
	\]
	where $\mathrm{Re}(\alpha)\geq0$, or simply $\alpha>0$ if we assume real order, and $n:=\lfloor\mathrm{Re}(\alpha)\rfloor+1$ so that $n-1\leq\mathrm{Re}(\alpha)<n$.
	
	The weighted Caputo fractional derivative of a given function $f\in C^n(a,b)$, with a weight function $w\in C^n[a,b]$ and order $\alpha$ in $\mathbb{R}$ or $\mathbb{C}$, is defined by
	\[
	\prescript{C}{a}D^{\alpha}_{x;w(x)}f(x)=\prescript{RL}{a}I^{n-\alpha}_{x;w(x)}\left(\frac{\mathrm{d}}{\mathrm{d}x} + \frac{w^{\prime}(x)}{w(x)}\right)^nf(x),
	\]
	where $\mathrm{Re}(\alpha)\geq0$, or simply $\alpha>0$ if we assume real order, and $n:=\lfloor\mathrm{Re}(\alpha)\rfloor+1$ so that $n-1\leq\mathrm{Re}(\alpha)<n$.
\end{defn}

\subsection{Conjugation relations}

It is clear that the weighted Riemann--Liouville fractional integral with weight function $w$ is given by multiplying by $w$, applying the original Riemann--Liouville fractional integral to the same order, and then dividing by $w$ again. This gives rise to a conjugation relation for the weighted fractional integral, which we can also extend to weighted fractional derivatives in order to achieve the following result.

\begin{prop}\label{CRs}
	The weighted fractional differintegrals are conjugations of the original fractional differintegrals, as follows:
\begin{align*}
\prescript{RL}{a}I^{\alpha}_{x;w(x)}&= M_{w(x)}^{-1}  \circ\prescript{RL}{a}I^{\alpha}_x\circ M_{w(x)},\\
\prescript{RL}{a}D^{\alpha}_{x;w(x)}&=  M_{w(x)}^{-1}  \circ\prescript{RL}{a}D^{\alpha}_x\circ M_{w(x)},\\
\prescript{C}{a}D^{\alpha}_{x;w(x)}&=  M_{w(x)}^{-1}  \circ\prescript{C}{a}D^{\alpha}_x\circ M_{w(x)},
\end{align*}
	where the operator $M_{w(x)}$ acting on functions is defined by multiplication:
	\begin{equation}\label{Mdef}
	\big(M_{w(x)}f\big)(x)=w(x)f(x).
	\end{equation}
\end{prop}

\begin{proof} \allowdisplaybreaks
The first result $\prescript{RL}{a}I^{\alpha}_{x;w(x)}= M_{w(x)}^{-1}  \circ\prescript{RL}{a}I^{\alpha}_x\circ M_{w(x)}$ is clear. Since both types of weighted fractional derivatives are compositions of the weighted fractional integral with the operator $\frac{\mathrm{d}}{\mathrm{d}x} + \frac{w^{\prime}(x)}{w(x)}$ repeated $n$ times, it will suffice to show that this non-fractional operator also satisfies a conjugation relation, $\frac{\mathrm{d}}{\mathrm{d}x} + \frac{w^{\prime}(x)}{w(x)}=M_{w(x)}^{-1}\circ\frac{\mathrm{d}}{\mathrm{d}x}\circ M_{w(x)}$. This is easily proved using the product rule:
\begin{align*}
M_{w(x)}^{-1}\circ\frac{\mathrm{d}}{\mathrm{d}x}\circ M_{w(x)}f(x)&=M_{w(x)}^{-1}\circ\frac{\mathrm{d}}{\mathrm{d}x}\big(w(x)f(x)\big)=M_{w(x)}^{-1}\big(w(x)f'(x)+w'(x)f(x)\big) \\
&=f'(x)+\frac{w'(x)}{w(x)}f(x)=\left(\frac{\mathrm{d}}{\mathrm{d}x} + \frac{w^{\prime}(x)}{w(x)}\right)f(x).
\end{align*}
The results follow from composition of conjugation relations.
\end{proof}

The results of Proposition \ref{CRs} are extremely useful in the study of weighted fractional calculus, because now many fundamental results from the original fractional calculus can be extended immediately, simply by composition, to the corresponding results on the weighted fractional operators. It is often true that the best way to prove facts about a new mathematical object is to express it in terms of an older one: this is the same principle that gave rise to the series formula for fractional calculus with general analytic kernels and to the conjugation relation for fractional calculus with respect to functions.

\begin{prop}
The weighted Riemann--Liouville derivative is the analytic continuation of the weighted Riemann--Liouville integral in the complex variable $\alpha$, under the convention that integrals of negative order are derivatives of positive order:
\[
\prescript{RL}{a}D^{\alpha}_{x;w(x)}f(x)=\prescript{RL}{a}I^{-\alpha}_{x;w(x)}f(x),\qquad\mathrm{Re}(\alpha)\geq0.
\]
This fact allows both $\prescript{RL}{a}I^{\alpha}_{x;w(x)}f(x)$ and $\prescript{RL}{a}D^{\alpha}_{x;w(x)}f(x)$ to be defined for all values of $\alpha\in\mathbb{C}$, in the same way as for the original Riemann--Liouville differintegrals.
\end{prop}

\begin{proof}
This follows immediately from the conjugation relations of Proposition \ref{CRs} together with the corresponding analytic continuation result for Riemann--Liouville differintegrals given at Equation \eqref{RL:analcont}.
\end{proof}

\begin{prop}
	The weighted fractional differintegrals have semigroup properties as follows:
	\begin{align*}
		\prescript{RL}{a}I^{\alpha}_{x;w(x)} \prescript{RL}{a}I^{\beta}_{x;w(x)} f(x) &= \prescript{RL}{a}I^{\alpha+\beta}_{x;w(x)} f(x),\qquad\alpha\in\mathbb{C},\mathrm{Re}(\beta)>0; \\
		\left(\frac{\mathrm{d}}{\mathrm{d}x} + \frac{w^{\prime}(x)}{w(x)}\right)^n \prescript{RL}{a}D^{\alpha}_{x;w(x)} f(x) &= \prescript{RL}{a}D^{n+\alpha}_{x;w(x)} f(x),\qquad\alpha\in\mathbb{C},n\in\mathbb{N},
	\end{align*}
	where in both cases $f$ is any function such that the relevant expressions are well-defined.

Note that the operators labelled by $\alpha$ in both of these relations may be either fractional integrals or fractional derivatives, while the one labelled by $\beta$ must be a fractional integral.
\end{prop}

\begin{proof}
	This is an immediate consequence of Proposition \ref{CRs} (conjugation relations) with Lemma \ref{Lem:RLsemigroup} (Riemann--Liouville semigroup properties).
\end{proof}

\begin{prop}
The following composition properties are valid for weighted Riemann--Liouville and Caputo differintegrals in cases where semigroup properties are not:
\begin{align*}
\prescript{RL}{a}I^{\alpha}_{x;w(x)}\prescript{RL}{a}D^{\alpha}_{x;w(x)}f(x)&=f(x)-\sum_{k=1}^n\frac{(x-a)^{\alpha-k}}{\Gamma(\alpha-k+1)}\cdot\frac{w(a^+)}{w(x)}\cdot\lim_{x\rightarrow a^+}\prescript{RL}{a}D^{\alpha-k}_{x;w(x)}f(x); \\
\prescript{RL}{a}I^{\alpha}_{x;w(x)}\prescript{C}{a}D^{\alpha}_{x;w(x)}f(x)&=f(x)-\sum_{k=0}^{n-1}\frac{(x-a)^k}{k!}\cdot\frac{w(a^+)}{w(x)}\cdot\lim_{x\rightarrow a^+}\left(\frac{\mathrm{d}}{\mathrm{d}x} + \frac{w'(x)}{w(x)}\right)^kf(x),
\end{align*}
where in both cases $\alpha\in\mathbb{C}$ with $\mathrm{Re}(\alpha)>0$ and $n=\lfloor\mathrm{Re}(\alpha)\rfloor+1$ while $f$ is any function such that the relevant expressions are well-defined. We also have the following relationship between the weighted Riemann--Liouville and Caputo derivatives:
\begin{align*}
\prescript{C}{a}D^{\alpha}_{x;w(x)}f(x)&=\prescript{RL}{a}D^{\alpha}_{x;w(x)}f(x)-\sum_{k=0}^{n-1}\frac{(x-a)^{k-\alpha}}{\Gamma(k-\alpha+1)}\cdot\frac{w(a^+)}{w(x)}\cdot\lim_{x\rightarrow a^+}\left(\frac{\mathrm{d}}{\mathrm{d}x} + \frac{w'(x)}{w(x)}\right)^kf(x) \\
&=\prescript{RL}{a}D^{\alpha}_{x;w(x)}\left(f(x)-\sum_{k=0}^{n-1}\frac{(x-a)^k}{k!}\cdot\frac{w(a^+)}{w(x)}\cdot\lim_{x\rightarrow a^+}\left(\frac{\mathrm{d}}{\mathrm{d}x} + \frac{w'(x)}{w(x)}\right)^kf(x)\right),
\end{align*}
where $f\in AC^n(a,b)$ and $\alpha,n$ are as before.
\end{prop}

\begin{proof}
All of these relations follow directly from combining the conjugation relations of Proposition \ref{CRs} with the previous relations of Lemma \ref{Lem:RLcompos}. But the conjugation process is not quite so easy this time, as the final term (finite sum) is more tricky to deal with than single weighted fractional operators or compositions thereof. We verify the first of the stated relations carefully as follows:
\begin{align*}
\prescript{RL}{a}I^{\alpha}_{x;w(x)}\prescript{RL}{a}D^{\alpha}_{x;w(x)}f(x)&=\frac{1}{w(x)}\prescript{RL}{a}I^{\alpha}_x\prescript{RL}{a}D^{\alpha}_x\Big(w(x)f(x)\Big) \\
&=\frac{1}{w(x)}\left[w(x)f(x)-\sum_{k=1}^n\frac{(x-a)^{\alpha-k}}{\Gamma(\alpha-k+1)}\cdot\lim_{x\rightarrow a^+}\prescript{RL}{a}D^{\alpha-k}_x\Big(w(x)f(x)\Big)\right] \\
&=f(x)-\sum_{k=1}^n\frac{(x-a)^{\alpha-k}}{\Gamma(\alpha-k+1)w(x)}\cdot\lim_{x\rightarrow a^+}\Big(w(x)\cdot\prescript{RL}{a}D^{\alpha-k}_{x;w(x)}f(x)\Big) \\
&=f(x)-\sum_{k=1}^n\frac{(x-a)^{\alpha-k}w(a^+)}{\Gamma(\alpha-k+1)w(x)}\cdot\lim_{x\rightarrow a^+}\prescript{RL}{a}D^{\alpha-k}_{x;w(x)}f(x),
\end{align*}
which is the required result. Similarly for the other relations stated in the Proposition.
\end{proof}

\begin{prop} \label{Prop:specfunc}
The weighted Riemann--Liouville and Caputo differintegrals of certain functions are given as follows:
\begin{alignat*}{2}
\prescript{RL}{a}D^{\alpha}_{x;w(x)}\left(\frac{(x-a)^{\beta}}{w(x)}\right)&=\frac{\Gamma(\beta+1)}{\Gamma(\beta-\alpha+1)}\frac{(x-a)^{\beta-\alpha}}{w(x)},&&\qquad\alpha\in\mathbb{C},\mathrm{Re}(\beta)>-1; \\
\prescript{C}{a}D^{\alpha}_{x;w(x)}\left(\frac{E_{\alpha}\big(\omega(x-a)^{\alpha}\big)}{w(x)}\right)&=\omega\cdot\frac{E_{\alpha}\big(\omega(x-a)^{\alpha}\big)}{w(x)},&&\qquad\omega\in\mathbb{C},\mathrm{Re}(\alpha)>0,
\end{alignat*}
where $E_{\alpha}$ is the Mittag-Leffler function. Note that the operator in the first identity can be either a fractional integral or a fractional derivative, according to the sign of $\mathrm{Re}(\alpha)$.
\end{prop}

	\begin{proof}
The result follows directly from the conjugation relations given by Proposition \ref{CRs}, combined with the results given in Lemma \ref{Lem:RLCspecfunc}.
	\end{proof}
	
\begin{rem}
Note that the functions used in Proposition \ref{Prop:specfunc} are just two possible examples that could have been chosen. Any known result for Riemann--Liouville or Caputo differintegrals of any particular functions can now easily be extended to an analogous result on weighted differintegrals, with the functions divided by $w(x)$ on left and right sides of the identity.

Note also that, for any given smooth nonzero function $h(x)$ on $(a,b)$, we can now find a weighted Caputo-type fractional derivative which has this function as an eigenfunction, by choosing the weight function to be $w(x)=\frac{E_{\alpha}\big(\omega(x-a)^{\alpha}\big)}{h(x)}$ and using the second result of Proposition \ref{Prop:specfunc}.
\end{rem}

\subsection{Examples}

In this subsection, we discuss some particular choices of weight function $w(x)$ which lead to interesting special cases in the theory of weighted fractional calculus. We start with a trivial example and then proceed to two more examples which give rise to models of fractional calculus that are already well known and studied in the literature.

\begin{ex}
If $w(x)=k$ is a constant, then the operators of weighted fractional calculus, as given in Definition \ref{Def:wRL&wC}, are exactly the same as the original Riemann--Liouville and Caputo operators given in Definition \ref{Def:RL&C}. This is because the operator $M_{w(x)}$ in this case commutes with the operators of fractional differentiation and integration, and indeed with any linear functional operator.

Note that this is the only possible type of function $w(x)$ which returns the original Riemann--Liouville and Caputo operators, because it is the only case when the weighted 1st-order derivative $\frac{\mathrm{d}}{\mathrm{d}x} + \frac{w^{\prime}(x)}{w(x)}$ is the same as the original 1st-order derivative $\frac{\mathrm{d}}{\mathrm{d}x}$.
\end{ex}

\begin{rem}
More generally, we can conclude that, for a given weight function $w(x)=w_0(x)$, any constant multiple $w(x)=kw_0(x)$ will give rise to exactly the same operators of weighted fractional calculus as $w_0(x)$. In fact, this implication goes both ways: two weight functions $w_1(x)$ and $w_2(x)$ give rise to the same weighted operators if and only if they are constant multiples of each other. Therefore, we can say that the class of weighted fractional calculi bijects with the space of possible weight functions quotiented by the equivalence relation of constant multiplication.
\end{rem}

\begin{ex} \label{Ex:tempered}
If $w(x)=e^{\beta x}$ is an exponential function, then the operators of weighted fractional calculus, as given in Definition \ref{Def:wRL&wC}, are precisely those of tempered fractional calculus, defined \cite{meerschaert-sabzikar-chen,li-deng-zhao} as follows:
\begin{align*}
	\prescript{T}{a}I^{\alpha,\beta}_xf(x)&=\frac{1}{\Gamma(\alpha)}\int_a^x(x-t)^{\alpha-1}e^{-\beta(x-t)}f(t)\,\mathrm{d}t,\qquad\beta\in\mathbb{C},\mathrm{Re}(\alpha)>0; \\
	\prescript{TR}{a}D^{\alpha,\beta}_xf(x)&=\left(\frac{\mathrm{d}}{\mathrm{d}x} + \beta\right)^n\prescript{T}{a}I^{n-\alpha,\beta}_xf(x),\qquad\beta\in\mathbb{C},\mathrm{Re}(\alpha)\geq0; \\
	\prescript{TC}{a}D^{\alpha,\beta}_xf(x)&=\prescript{T}{a}I^{n-\alpha,\beta}_x\left(\frac{\mathrm{d}}{\mathrm{d}x} + \beta\right)^nf(x),\qquad\beta\in\mathbb{C},\mathrm{Re}(\alpha)\geq0.
\end{align*}
Therefore, tempered fractional calculus is within the class of weighted fractional calculus. It is also known \cite{fernandez-ustaoglu} that tempered fractional calculus is within the class of fractional calculus with analytic kernels. In fact, the intersection between these two general classes of fractional operators consists only of tempered fractional calculus, as we show in the following Theorem.
\end{ex}

\begin{thm}
The only operator which is a member of both the class of fractional integrals with general analytic kernels and the class of weighted fractional integrals is the tempered fractional integral. In other words, if
\[
\mathcal{O}=\prescript{A}{a}I^{\alpha,\beta}_x=\prescript{RL}{a}I^{\alpha}_{x;w(x)},
\]
with an analytic function $A$ and a real continuous function $w$, then
\[
\mathcal{O}=\prescript{T}{a}I^{\alpha,\kappa}_x
\]
for some constant $\kappa\in\mathbb{R}$.
\end{thm}

\begin{proof}
We start with an operator $\mathcal{O}$ which is in both classes, and write the integrals explicitly as follows:
\[
\mathcal{O}f(x)=\int_a^x(x-t)^{\alpha-1}A\big((x-t)^{\beta}\big)f(t)\,\mathrm{d}t=\frac{1}{\Gamma(\alpha)w(x)}\int_a^x(x-t)^{\alpha-1}w(t)f(t)\,\mathrm{d}t,
\]
for all $f\in L^1(a,b)$ and $x\in(a,b)$. Therefore, we must have
\[
\Gamma(\alpha)A\big((x-t)^{\beta}\big)=\frac{w(t)}{w(x)},\qquad a<t<x<b.
\]
The left-hand side is a function of $x-t$; let us write $q(z)=\Gamma(\alpha)A(z^{\beta})$, noting that this is a smooth function on the real interval $z\in[0,b-a]$. Without loss of generality, let us assume $a=0$ as the lower end of the interval. Now we have $q(x-t)=\frac{w(t)}{w(x)}$ for $0<t<x<b$. Letting $t\to0$, we find that $q(x)=\frac{w(0)}{w(x)}$, so the equation becomes
\[
\frac{w(0)}{w(x-t)}=\frac{w(t)}{w(x)},\qquad 0<t<x<b.
\]
Writing $W(z)=\frac{w(z)}{w(0)}$ (a constant multiple of the original function $w$), we obtain Cauchy's functional equation:
\[
W(x-t)W(t)=W(x),\qquad 0<t<x<b.
\]
Since we assumed $w$ real and continuous, this means it must be a constant times an exponential function, $w(x)=w(0)e^{\kappa x}$ for some constant $\kappa\in\mathbb{R}$, which gives the result.
\end{proof}

\begin{ex} \label{Ex:Kober}
If $w(x)=x^{\eta}$ is a power function, then the operators of weighted fractional calculus, as given in Definition \ref{Def:wRL&wC}, are almost exactly those of the so-called Kober--Erd\'elyi (or simply Kober) fractional calculus, defined in \cite[\S18.1]{samko-kilbas-marichev} and \cite[\S2.6]{kilbas-srivastava-trujillo} as follows:
\[
\prescript{K}{a}I^{\alpha;\eta}_xf(x)=\frac{x^{-\alpha-\eta}}{\Gamma(\alpha)}\int_a^x(x-t)^{\alpha-1}t^{\eta}f(t)\,\mathrm{d}t,\qquad\mathrm{Re}(\alpha)>0,\mathrm{Re}(\eta)>0.
\]
(Usually, the Kober--Erd\'elyi operator is written with lower bound $a=0$, but we see no reason here to restrict to this particular case, as it can easily be written with the general constant of integration $a$.) It is clear that we have
\[
\prescript{K}{a}I^{\alpha;\eta}_xf(x)=x^{-\alpha}\cdot\prescript{RL}{a}I^{\alpha}_{x;x^{\eta}}f(x).
\]
This fact was already noted in \cite[Eq. (2.6.15)]{kilbas-srivastava-trujillo}, using a different notation of $M_{\eta}$ operators to give a relationship between the Kober--Erd\'elyi integral and the Riemann--Liouville integral.

A closely related, but more general, operator is the so-called Erd\'elyi--Kober fractional integral, which is also discussed in \cite[\S18.1]{samko-kilbas-marichev} and \cite[\S2.6]{kilbas-srivastava-trujillo}. However, this involves a fractional power of $(x^{\sigma}-t^{\sigma})$, not only $(x-t)$, so it is not easily related to a special case of weighted fractional calculus as given in Definition \ref{Def:wRL&wC}. In fact, the appearance of this operator has elements of both weighted fractional calculus and fractional calculus with respect to functions, as given in Definition \ref{Def:WRTF}, and later we shall see that it is a special case of the larger class of operators described in Section \ref{Sec:WFCwrtf} as weighted fractional calculus with respect to functions.
\end{ex}

\subsection{Laplace transform and convolution} 

In this subsection, we discuss a weighted integral transform, a very simple variation on the Laplace transform, which can be used to solve linear fractional differential equations involving weighted Riemann--Liouville and Caputo fractional derivatives. Related to this transform, we shall introduce a weighted convolution operation, which has a conjugation relation with the classical Laplace-type convolution. Using the conjugation approach and well-known facts about the classical Laplace transform and convolution, we can easily prove the fundamental properties of the weighted Laplace transform and weighted convolution.

\begin{defn}
	Let $ f : [0, \infty) \to \mathbb{C} $ be a real-valued or complex-valued function. Then the weighted Laplace transform of $ f $ with a weight function $ w $ is defined by
	\begin{equation} \label{WLT}
		\mathcal{L}_{w(x)} \left\{ f(x) \right\}=F(s)=\int_{0}^{\infty} e^{-s x} w(x) f(x) \,\mathrm{d}x,
	\end{equation}
	for any $s\in\mathbb{C}$ such that this is a convergent integral.
\end{defn}

It is clear that the $w$-weighted Laplace transform of $f$ is precisely the usual Laplace transform of the function $w(x)f(x)$: in operational calculus terms, this means that
	\begin{equation} \label{WLT:comp}
		\mathcal{L}_{w(x)}=\mathcal{L}\circ M_{w(x)},
	\end{equation}
where $M_{w(x)}$ is the multiplication operator defined by \eqref{Mdef} above.

Therefore, $F(s)$ exists for $\mathrm{Re}(s)>c$ if $w$ and $f$ are such that $w(x)f(x)$ is of exponential order $c$: for example, if $w\in L^{\infty}[0,\infty)$ and $f$ is of exponential order $c$, or if $w$ is of exponential order $c_1$ and $f$ is of exponential order $c_2$ with $c_1+c_2=c$.

As corollaries of the relation \eqref{WLT:comp}, we mention the following results.

\begin{cor}\label{IWLT:cor}
The inverse weighted Laplace transform exists for any function which has a classical inverse Laplace transform, and it may be written as follows:
	\begin{equation}
		\label{IWLT:comp}
		\mathcal{L}_{w(x)}^{-1}=M_{w(x)}^{-1}\circ\mathcal{L}^{-1},
	\end{equation}
	or in other words
	\begin{equation} \label{inverseWL}
		\mathcal{L}^{-1}_{w(x)} \left\{ F(s)  \right\} =\frac{{1}}{{2 \pi i w(x)}} \int_{c-i\infty}^{c+i\infty} e^{sx }  F(s) \,\mathrm{d}s.
	\end{equation}
\end{cor}

\begin{cor}\label{Coroll:genfn}
	If $f$ is a function which has a classical Laplace transform $F(s)$, then the weighted Laplace transform of $\left(M_{w(x)}^{-1}f\right)(x)=\frac{f(x)}{w(x)}$ is also $F(s)$:
	\[
	\mathcal{L}\{f(x)\}=F(s)\quad\Rightarrow\quad\mathcal{L}_{w(x)}\left\{\frac{f(x)}{w(x)}\right\}=F(s).
	\]
\end{cor}

The reason for introducing a weighted Laplace transform, even when it is so similar to the classical Laplace transform that all results concerning it are trivially proved, is that it has a natural relationship with the operators of weighted fractional calculus. In the following theorem, we find the weighted Laplace transforms of the weighted Riemann--Liouville and Caputo fractional differintegrals.

\begin{thm} \label{inttrans}
	Let $ \alpha >0 $ and let $f$ be a continuous function on $[0,\infty)$ which is of $w$-weighted exponential order, where $w$ is a continuous weight function. Then we have the following results.
	\begin{enumerate}
	\item \[\mathcal{L}_{w(x)} \left\{ \left( 	\prescript{RL}{0}I^{\alpha}_{x;w(x)}f\right)(x) \right\}= s^{-\alpha}\mathcal{L}_{w(x)} \left\{f(x)\right\}.\]
	\item Let $ n-1\leq\mathrm{Re}(\alpha)<n\in\mathbb{Z}^+ $, and assume that $\prescript{RL}{0}D^{\alpha}_{x;w(x)}f$ is continuous on $[0,\infty)$ and of $w$-weighted exponential order. Then
	\[\mathcal{L}_{w(x)}\left\{\left(\prescript{RL}{0}D^{\alpha}_{x;w(x)}f\right)(x) \right\} = s^{\alpha}\mathcal{L}_{w(x)}\{f(x)\} - w(0^+)\sum_{i=0}^{n-1} s^{n-i-1}\left(\prescript{RL}{0}I^{n-i-\alpha}_{x;w(x)}f\right)(0^+).\]
	\item Let $ n-1\leq\mathrm{Re}(\alpha)<n\in\mathbb{Z}^+ $, and assume that $\prescript{RL}{0}D^{n}_{x;w(x)}f$ is continuous on $[0,\infty)$ and of $w$-weighted exponential order. Then
	\[\mathcal{L}_{w(x)} \left\{ \left(	\prescript{C}{0}D^{\alpha}_{x;w(x)}f\right)(x) \right\} =s^{\alpha} \mathcal{L}_{w(x)}\{f(x)\} - w(0^+)\sum_{i=0}^{n-1} s^{\alpha-i-1}\left[\left(\frac{\mathrm{d}}{\mathrm{d}x} + \frac{w^{\prime}(x)}{w(x)}\right)^if\right](0^+).\]
	\end{enumerate}
\end{thm}

\begin{proof}
All of these results follow from combining the results of Proposition \ref{CRs} and Equation \eqref{WLT:comp} with the following well-known facts on classical Laplace transforms of fractional integrals and derivatives:
\begin{align*}
\mathcal{L} \left\{ \left( 	\prescript{RL}{0}I^{\alpha}_{x}f\right)(x) \right\}&= s^{-\alpha}\mathcal{L} \left\{ f(x) \right\}, \\
\mathcal{L}\left\{\left(\prescript{RL}{0}D^{\alpha}_{x}f\right)(x) \right\} &= s^{\alpha}\mathcal{L}\{f(x)\} - \sum_{i=0}^{n-1} s^{n-i-1}\left(\prescript{}{0}I^{n-i-\alpha}_{x}f\right)(0^+), \\
\mathcal{L} \left\{ \left(	\prescript{C}{0}D^{\alpha}_{x}f\right)(x) \right\} &=s^{\alpha} \mathcal{L}\{f(x)\} - \sum_{i=0}^{n-1} s^{\alpha-i-1}f^{(i)}(0^+).
\end{align*}
Note that the given conditions on $f$ imply that (in the second case, Riemann--Liouville derivatives) the function $\prescript{RL}{0}I^{n-\alpha}_{x;w(x)}f$ and all its weighted derivatives from $1$st to $n$th order, namely $\prescript{}{0}D^{\alpha-n+1}_{x;w(x)}f$, $\ldots$, $\prescript{RL}{0}D^{\alpha-1}_{x;w(x)}f$, $\prescript{RL}{0}D^{\alpha}_{x;w(x)}f$, are continuous on $[0,\infty)$ and of $w$-weighted exponential order, while (in the third case, Caputo derivatives) the function $f$ and all its weighted derivatives from $1$st to $n$th order, namely $\prescript{RL}{0}D^{1}_{x;w(x)}f$, $ \prescript{RL}{0}D^{2}_{x;w(x)}f$, $\ldots$, $\prescript{RL}{0}D^{n}_{x;w(x)}f$, are continuous on $[0,\infty)$ and of $w$-weighted exponential order, and so is $\prescript{C}{0}D^{\alpha}_{x;w(x)}f $. This is because continuity and exponential boundedness are preserved by integration; we have fixed our assumptions to be only at the highest order of differentiation.
\end{proof}

Now, we consider a $ w $-weighted convolution operation between two functions, which relates naturally to the weighted Laplace transform.

\begin{defn} 
The $ w $-weighted convolution of two real-valued or complex-valued functions $f,g:[0,\infty)\to\mathbb{C}$ is the function $f\ast_{w(x)} g$ defined by 
	\begin{equation}\label{convolution}
		\Big(f*_{w(x)} g\Big)(x)= \frac{1}{w(x)}\int_{0}^{x} w(x-t)f\left( x-t \right)  w(t)g(t)  \,\mathrm{d}t.
	\end{equation}
\end{defn}

It is clear that this weighted convolution is closely related to classical convolution via the following formula:
\[
\Big(f*_{w(x)} g\Big)(x)= \frac{1}{w(x)}\big(wf\big)*\big(wg\big)(x),
\]
or in other words, using the multiplication operator defined in \eqref{Mdef},
	\begin{equation}\label{RSbwWCnC}
		f\ast_{w(x)} g = M_{w(x)}^{-1} \Big( \left( M_{w(x)} f \right) \ast  \left( M_{w(x)} g \right) \Big).
	\end{equation}
Since both classical and $ w $-weighted convolutions are binary operations acting on pairs of functions, we can use the alternative notations $\ast(f,g)$ and $\ast_{w(x)}(f,g)$, instead of $f\ast g$ and $f\ast_{w(x)} g$, in order to rewrite \eqref{RSbwWCnC} as a conjugation relation of operators:
	\begin{equation*}
		\ast_{w(x)} = M_{w(x)}^{-1} \circ \ast \circ \Big( M_{w(x)},M_{w(x)} \Big) \quad\Rightarrow\quad
		\ast_{w(x)}\left( f,g \right)  = M_{w(x)}^{-1}  \Big( \ast \left(M_{w(x)}f,M_{w(x)}g\right)  \Big).
	\end{equation*}

These relations, combined with the operational relation \eqref{WLT:comp} for the weighted Laplace transforms, give rise to the following important corollary, an analogue of the classical Laplace convolution theorem.

\begin{cor}
	If $f,g:[0,X]\to\mathbb{C}$ are piecewise continuous and their products with $w$ are of exponential order $c>0 $, then
	\begin{equation*}
		\mathcal{L}_{w(x)} \left\{ f\ast_{w(x)} g\right\}=\mathcal{L}_{w(x)}\{f\}\mathcal{L}_{w(x)}\{g\}.
	\end{equation*}
\end{cor}

In the following theorem, we use the weighted Laplace transform method to establish a regularity condition for the solutions to weighted fractional differential equations of the following type:
\begin{align}
	\prescript{C}{0}D^{\alpha}_{x;w(x)}\boldsymbol{y}(x)&=A\boldsymbol{y}(x)+\boldsymbol{g}(x), \qquad x\geq0, \label{de}
	\\ \boldsymbol{y}(0)&= \boldsymbol{\eta},  \label{ic}
\end{align}
where $0<\alpha<1$ is fixed and $ A = \left(a_{ij}\right) $ is an $n\times n$ constant matrix and $\boldsymbol{g}$ is a continuous $n$-dimensional vector-valued function and $\boldsymbol{\eta}$ is a constant $n$-dimensional vector.

\begin{thm}
	Assume that the system $\eqref{de}-\eqref{ic}$ has a unique continuous solution $ \boldsymbol{y} $. If $ \boldsymbol{g}$ is continuous on $[0,\infty) $ and $ w $-weighted exponentially bounded, then $ \boldsymbol{y} $ and $ \prescript{C}{0}D^{\alpha}_{x;w(x)}\boldsymbol{y} $ are both $ w $-weighted exponentially bounded too.
\end{thm}
\begin{proof}
	Using the result proved in \cite[Theorem 3.1]{Kexue} together with the conjugation relations given by Proposition \ref{CRs}, we obtain the required result.
\end{proof}

\section{Weighted fractional calculus with respect to functions} \label{Sec:WFCwrtf}

\begin{defn}[{\cite{agrawal1,jarad-abdeljawad-shah}}] \label{Def:WW}
	Let $\phi:[a,b]\to\mathbb{R}$ be a strictly increasing $C^1$ function, so that $\phi'>0$ everywhere, and let $w\in L^{\infty}(a,b)$ be a weight function. The $w$-weighted Riemann--Liouville fractional integral with respect to $\phi$ of a given function $f\in L^1(a,b)$, to order $\alpha$ in $\mathbb{R}$ or $\mathbb{C}$, is defined by
	\[
	\prescript{RL}{a}I^{\alpha}_{\phi(x);w(x)}f(x)=\frac{1}{\Gamma(\alpha)w(x)}\int_a^x\big(\phi(x)-\phi(t)\big)^{\alpha-1}w(t)f(t)\phi'(t)\,\mathrm{d}t,\qquad x\in(a,b),
	\]
	where we require $\mathrm{Re}(\alpha)>0$, or simply $\alpha>0$ if we assume real order.
	
	Assuming sufficient differentiability conditions on $\phi$ and $w$, we can also define the $w$-weighted Riemann--Liouville fractional derivative with respect to $\phi$ of a given function $f\in AC^n_{\phi}(a,b)$ and the $w$-weighted Caputo fractional derivative with respect to $\phi$ of a given function $f\in C^n_{\phi}(a,b)$, to order $\alpha$ in $\mathbb{R}$ or $\mathbb{C}$, as follows:
	\begin{align*}
	\prescript{RL}{a}D^{\alpha}_{\phi(x);w(x)}f(x)&=\Big(D_{\phi(x);w(x)}\Big)^n\prescript{RL}{a}I^{n-\alpha}_{\phi(x);w(x)}f(x), \\
	\prescript{C}{a}D^{\alpha}_{\phi(x);w(x)}f(x)&=\prescript{RL}{a}I^{n-\alpha}_{\phi(x);w(x)}\Big(D_{\phi(x);w(x)}\Big)^nf(x),
	\end{align*}
	where $\mathrm{Re}(\alpha)\geq0$, or simply $\alpha>0$ if we assume real order, and $n:=\lfloor\mathrm{Re}(\alpha)\rfloor+1$ so that $n-1\leq\mathrm{Re}(\alpha)<n$, and where the first-order operator $D_{\phi(x);w(x)}$ is defined by
	\[
	D_{\phi(x);w(x)}f(x)=\frac{1}{w(x)\phi'(x)}\cdot\frac{\mathrm{d}}{\mathrm{d}x}\big(w(x)f(x)\big),
	\]
	or equivalently $D_{\phi(x);w(x)}=\frac{1}{\phi'(x)}\cdot\left(\frac{\mathrm{d}}{\mathrm{d}x}+\frac{w'(x)}{w(x)}\right)$.
\end{defn}

These operators were introduced by Agrawal \cite{agrawal1,agrawal2} and also studied by Jarad et al \cite{jarad-abdeljawad-shah}. Here we shall consider them from the operational viewpoint, using the notion of conjugation to connect them directly to the original Riemann--Liouville and Caputo operators and thence to derive many results in an easy way without the unnecessary calculations of some previous works \cite{jarad-abdeljawad-shah}.

\subsection{Conjugation relations}

\begin{thm}[{\cite[\S3.3]{kolokoltsov}}] \label{Thm:WWconjug}
	The weighted fractional differintegrals with respect to functions may be written as conjugations of the original fractional differintegrals as follows:
\begin{align*}
\prescript{RL}{a}I^{\alpha}_{\phi(x);w(x)}&= M_{w(x)}^{-1}\circ Q_{\phi}\circ\prescript{RL}{\phi(a)}I^{\alpha}_x\circ Q_{\phi}^{-1}\circ M_{w(x)},\\
\prescript{RL}{a}D^{\alpha}_{\phi(x);w(x)}&=M_{w(x)}^{-1}\circ Q_{\phi}\circ\prescript{RL}{\phi(a)}D^{\alpha}_x\circ Q_{\phi}^{-1}\circ M_{w(x)},\\
\prescript{C}{a}D^{\alpha}_{\phi(x);w(x)}&=M_{w(x)}^{-1}\circ Q_{\phi}\circ\prescript{C}{\phi(a)}D^{\alpha}_x\circ Q_{\phi}^{-1}\circ M_{w(x)},
\end{align*}
	where the multiplication operator $M_{w(x)}$ is defined by \eqref{Mdef} and the composition operator $Q_{\phi}$ is defined in Lemma \ref{Lem:WRTFconjug}.
\end{thm}

\begin{proof}
It is clear that the $w$-weighted Riemann--Liouville fractional integral with respect to $\phi$ is given by multiplying by $w$, applying the (unweighted) Riemann--Liouville fractional integral with respect to $\phi$ to the same order, and then dividing by $w$ again. This gives a conjugation relation between the weighted fractional integral with respect to $\phi$ and the original fractional integral with respect to $\phi$, which can be combined with Lemma \ref{Lem:WRTFconjug} as follows:
\[
\prescript{RL}{a}I^{\alpha}_{\phi(x);w(x)}= M_{w(x)}^{-1}\circ \prescript{RL}{a}I^{\alpha}_{\phi(x)}\circ M_{w(x)}=M_{w(x)}^{-1}\circ\Big(Q_\phi\circ\prescript{RL}{\phi(a)}I^{\alpha}_x\circ Q_\phi^{-1}\Big)\circ M_{w(x)}.
\]
Both types of fractional derivatives are compositions of this fractional integral operator with the first-order operator $D_{\phi(x);w(x)}=\frac{1}{\phi'(x)}\cdot\left(\frac{\mathrm{d}}{\mathrm{d}x}+\frac{w'(x)}{w(x)}\right)$ repeated $n$ times, so it will suffice to show that $D_{\phi(x);w(x)}$ also satisfies a conjugation relation, which is easily proved using the product rule and chain rule:
\begin{align*}
M_{w(x)}^{-1}\circ Q_\phi\circ\frac{\mathrm{d}}{\mathrm{d}x}\circ Q_\phi^{-1}\circ M_{w(x)}f(x)&=M_{w(x)}^{-1}\circ\left(\frac{1}{\phi'(x)}\cdot\frac{\mathrm{d}}{\mathrm{d}x}\right)\circ M_{w(x)}f(x) \\
&=M_{w(x)}^{-1}\left(\frac{1}{\phi'(x)}\cdot\frac{\mathrm{d}}{\mathrm{d}x}\Big(w(x)f(x)\Big)\right) \\
&=\frac{1}{w(x)\phi'(x)}\Big(w(x)f'(x)+w'(x)f(x)\Big) \\
&=\frac{1}{\phi'(x)}\left(\frac{\mathrm{d}}{\mathrm{d}x}+\frac{w'(x)}{w(x)}\right)f(x).
\end{align*}
The results follow from composition of conjugation relations.
\end{proof}

\begin{rem} \label{Rem:whichway}
The above result shows that the operators of Definition \ref{Def:WW} are obtained by starting with the original operators of fractional calculus, applying $Q_{\phi}$-conjugation to make the operators with respect to functions, and then applying $M_{w}$-conjugation to make them weighted. There is a definite order to the modifications: we are using the weighted versions of operators with respect to functions. What about the with-respect-to-functions versions of weighted operators, as mentioned briefly in \cite[Eq. (3.36)--3.38)]{kolokoltsov}?

In fact, these two potential general classes of operators come to the same thing, because of the relation $M_{w(x)}\circ Q_{\phi}^{-1}=Q_{\phi}^{-1}\circ M_{w(\phi(x))}$. It can be checked by some elementary calculations that
\[
Q_\phi\circ M_{w(x)}^{-1}\circ\frac{\mathrm{d}}{\mathrm{d}x}\circ M_{w(x)}\circ Q_\phi^{-1}f(x)=\frac{1}{\phi'(x)}\left(\frac{\mathrm{d}}{\mathrm{d}x}+\frac{(w\circ\phi)'(x)}{w\circ\phi(x)}\right)f(x),
\]
confirming that the class of weighted fractional operators with respect to functions is the same class regardless of which order we apply the weighted and with-respect-to-functions aspects of the operators. To write the alternative conjugation relations explicitly, we have
\begin{align*}
\prescript{RL}{a}I^{\alpha}_{\phi(x);w(x)}&=Q_{\phi}\circ M_{w(\phi^{-1}(x))}^{-1}\circ\prescript{RL}{\phi(a)}I^{\alpha}_x\circ M_{w(\phi^{-1}(x))}\circ Q_{\phi}^{-1},\\
\prescript{RL}{a}D^{\alpha}_{\phi(x);w(x)}&=Q_{\phi}\circ M_{w(\phi^{-1}(x))}^{-1}\circ\prescript{RL}{\phi(a)}D^{\alpha}_x\circ M_{w(\phi^{-1}(x))}\circ Q_{\phi}^{-1},\\
\prescript{C}{a}D^{\alpha}_{\phi(x);w(x)}&=Q_{\phi}\circ M_{w(\phi^{-1}(x))}^{-1}\circ\prescript{C}{\phi(a)}D^{\alpha}_x\circ M_{w(\phi^{-1}(x))}\circ Q_{\phi}^{-1},
\end{align*}
derived from the results of Theorem \ref{Thm:WWconjug}.
\end{rem}

The above conjugation results are very useful in understanding weighted fractional calculus with respect to functions. They were mentioned in \cite[Eq. (3.31)--(3.35)]{kolokoltsov}, but they were not used in \cite{jarad-abdeljawad-shah}, many of whose results can be proved much more quickly by directly using the corresponding results from classical fractional calculus together with the conjugation relations.

\begin{prop} \label{Prop:WWanalcont}
The $w$-weighted RL derivative with respect to $\phi$ is the analytic continuation, in the complex variable $\alpha$, of the $w$-weighted RL integral with respect to $\phi$, under the convention that integrals of negative order are derivatives of positive order:
\[
\prescript{RL}{a}D^{\alpha}_{\phi(x);w(x)}f(x)=\prescript{RL}{a}I^{-\alpha}_{\phi(x);w(x)}f(x),\qquad\mathrm{Re}(\alpha)\geq0.
\]
This fact allows both $\prescript{RL}{a}I^{\alpha}_{\phi(x);w(x)}f(x)$ and $\prescript{RL}{a}D^{\alpha}_{\phi(x);w(x)}f(x)$ to be defined for all values of $\alpha\in\mathbb{C}$, in the same way as for the original Riemann--Liouville differintegrals.
\end{prop}

\begin{proof}
This follows directly from the conjugation relations of Theorem \ref{Thm:WWconjug} together with the corresponding analytic continuation result for Riemann--Liouville differintegrals given at Equation \eqref{RL:analcont}.
\end{proof}

\begin{prop} \label{Prop:WWsemigroup}
	The weighted fractional differintegrals with respect to functions have semigroup properties as follows:
	\begin{align*}
		\prescript{RL}{a}I^{\alpha}_{\phi(x);w(x)} \prescript{RL}{a}I^{\beta}_{\phi(x);w(x)} f(x) &= \prescript{RL}{a}I^{\alpha+\beta}_{\phi(x);w(x)} f(x),\qquad\alpha\in\mathbb{C},\mathrm{Re}(\beta)>0; \\
		\Big(D_{\phi(x);w(x)}\Big)^n \prescript{RL}{a}D^{\alpha}_{\phi(x);w(x)} f(x) &= \prescript{RL}{a}D^{n+\alpha}_{\phi(x);w(x)} f(x),\qquad\alpha\in\mathbb{C},n\in\mathbb{N},
	\end{align*}
	where in both cases $f$ is any function such that the relevant expressions are well-defined.

Note that the operators labelled by $\alpha$ in both of these relations may be either fractional integrals or fractional derivatives, while the one labelled by $\beta$ must be a fractional integral.
\end{prop}

\begin{proof}
	This is an immediate consequence of Theorem \ref{Thm:WWconjug} (conjugation relations) with Lemma \ref{Lem:RLsemigroup} (Riemann--Liouville semigroup properties).
\end{proof}

\begin{prop}
The following composition properties are valid for weighted Riemann--Liouville and Caputo differintegrals with respect to functions in cases where semigroup properties are not:
\begin{align*}
\prescript{RL}{a}I^{\alpha}_{\phi(x);w(x)}\prescript{RL}{a}D^{\alpha}_{\phi(x);w(x)}f(x)&=f(x)-\sum_{k=1}^n\frac{\big(\phi(x)-\phi(a)\big)^{\alpha-k}}{\Gamma(\alpha-k+1)}\cdot\frac{w(a^+)}{w(x)}\cdot\lim_{x\rightarrow a^+}\prescript{RL}{a}D^{\alpha-k}_{\phi(x);w(x)}f(x); \\
\prescript{RL}{a}I^{\alpha}_{x;w(x)}\prescript{C}{a}D^{\alpha}_{x;w(x)}f(x)&=f(x)-\sum_{k=0}^{n-1}\frac{\big(\phi(x)-\phi(a)\big)^k}{k!}\cdot\frac{w(a^+)}{w(x)}\cdot\lim_{x\rightarrow a^+}\Big(D_{\phi(x);w(x)}\Big)^kf(x),
\end{align*}
where in both cases $\alpha\in\mathbb{C}$ with $\mathrm{Re}(\alpha)>0$ and $n=\lfloor\mathrm{Re}(\alpha)\rfloor+1$ while $f$ is any function such that the relevant expressions are well-defined. We also have the following relationship between the weighted Riemann--Liouville and Caputo derivatives with respect to functions:
\begin{align*}
\prescript{C}{a}D^{\alpha}_{\phi(x);w(x)}f(x)&=\prescript{RL}{a}D^{\alpha}_{\phi(x);w(x)}f(x)-\sum_{k=0}^{n-1}\frac{\big(\phi(x)-\phi(a)\big)^{k-\alpha}}{\Gamma(k-\alpha+1)}\cdot\frac{w(a^+)}{w(x)}\cdot\lim_{x\rightarrow a^+}\Big(D_{\phi(x);w(x)}\Big)^kf(x) \\
&=\prescript{RL}{a}D^{\alpha}_{\phi(x);w(x)}\left(f(x)-\sum_{k=0}^{n-1}\frac{\big(\phi(x)-\phi(a)\big)^k}{k!}\cdot\frac{w(a^+)}{w(x)}\cdot\lim_{x\rightarrow a^+}\Big(D_{\phi(x);w(x)}\Big)^kf(x)\right),
\end{align*}
where $f\in AC^n_{\phi}(a,b)$ and $\alpha,n$ are as before.
\end{prop}

\begin{proof}
We verify the first of the stated relations as follows:
\begin{align*}
\prescript{RL}{a}I^{\alpha}_{\phi(x);w(x)}\prescript{RL}{a}D^{\alpha}_{\phi(x);w(x)}f(x)&=M_{w(x)}^{-1}\circ Q_{\phi}\circ\prescript{RL}{\phi(a)}I^{\alpha}_x\circ\prescript{RL}{\phi(a)}D^{\alpha}_x\circ Q_{\phi}^{-1}\circ M_{w(x)}f(x) \\
&\hspace{-3.5cm}=M_{w(x)}^{-1}\circ Q_{\phi}\left[\prescript{RL}{\phi(a)}I^{\alpha}_x\prescript{RL}{\phi(a)}D^{\alpha}_x\Big(w\big(\phi^{-1}(x)\big)f\big(\phi^{-1}(x)\big)\Big)\right] \\
&\hspace{-3.5cm}=M_{w(x)}^{-1}\circ Q_{\phi}\Bigg[w\big(\phi^{-1}(x)\big)f\big(\phi^{-1}(x)\big) \\ &\hspace{-1cm}-\sum_{k=1}^n\frac{\big(x-\phi(a)\big)^{\alpha-k}}{\Gamma(\alpha-k+1)}\cdot\lim_{x\rightarrow\phi(a)^+}\prescript{RL}{\phi(a)}D^{\alpha-k}_x\Big(w\big(\phi^{-1}(x)\big)f\big(\phi^{-1}(x)\big)\Big)\Bigg] \\
&\hspace{-3.5cm}=\frac{1}{w(x)}\left[w(x)f(x)-\sum_{k=1}^n\frac{\big(\phi(x)-\phi(a)\big)^{\alpha-k}}{\Gamma(\alpha-k+1)}\cdot\lim_{x\rightarrow\phi(a)^+}\prescript{RL}{\phi(a)}D^{\alpha-k}_x\Big(w\big(\phi^{-1}(x)\big)f\big(\phi^{-1}(x)\big)\Big)\right] \\
&\hspace{-3.5cm}=f(x)-\sum_{k=1}^n\frac{\big(\phi(x)-\phi(a)\big)^{\alpha-k}}{\Gamma(\alpha-k+1)w(x)}\cdot\lim_{x\rightarrow\phi(a)^+}\prescript{RL}{\phi(a)}D^{\alpha-k}_x\circ Q_{\phi}^{-1}\circ M_{w(x)}f(x) \\
&\hspace{-3.5cm}=f(x)-\sum_{k=1}^n\frac{\big(\phi(x)-\phi(a)\big)^{\alpha-k}}{\Gamma(\alpha-k+1)w(x)}\cdot\lim_{x\rightarrow a^+}Q_{\phi}\circ\prescript{RL}{\phi(a)}D^{\alpha-k}_x\circ Q_{\phi}^{-1}\circ M_{w(x)}f(x) \\
&\hspace{-3.5cm}=f(x)-\sum_{k=1}^n\frac{\big(\phi(x)-\phi(a)\big)^{\alpha-k}}{\Gamma(\alpha-k+1)}\cdot\frac{w(a^+)}{w(x)}\cdot\lim_{x\rightarrow a^+}\prescript{RL}{a}D^{\alpha-k}_{\phi(x);w(x)}f(x),
\end{align*}
which is the required result. Similar manipulations can be used to prove all of the other stated relations, starting from the corresponding relations for Riemann--Liouville and Caputo differintegrals; we omit the straightforward details.
\end{proof}

\begin{prop} \label{Prop:WWspecfunc}
The weighted Riemann--Liouville and Caputo differintegrals of certain functions with respect to another function are given as follows:
\begin{alignat*}{2}
\prescript{RL}{a}D^{\alpha}_{\phi(x);w(x)}\left(\frac{\big(\phi(x)-\phi(a)\big)^{\beta}}{w(x)}\right)&=\frac{\Gamma(\beta+1)}{\Gamma(\beta-\alpha+1)}\frac{\big(\phi(x)-\phi(a)\big)^{\beta-\alpha}}{w(x)},&&\qquad\alpha\in\mathbb{C},\mathrm{Re}(\beta)>-1; \\
\prescript{C}{a}D^{\alpha}_{\phi(x);w(x)}\left(\frac{E_{\alpha}\Big(\omega\big(\phi(x)-\phi(a)\big)^{\alpha}\Big)}{w(x)}\right)&=\omega\cdot\frac{E_{\alpha}\Big(\omega\big(\phi(x)-\phi(a)\big)^{\alpha}\Big)}{w(x)},&&\qquad\omega\in\mathbb{C},\mathrm{Re}(\alpha)>0,
\end{alignat*}
where $E_{\alpha}$ is the Mittag-Leffler function. Note that the operator in the first identity can be either a fractional integral or a fractional derivative, according to the sign of $\mathrm{Re}(\alpha)$.
\end{prop}

	\begin{proof}
The result follows directly from the conjugation relations of Theorem \ref{Thm:WWconjug}, combined with the results of Lemma \ref{Lem:RLCspecfunc}.
	\end{proof}
	
\begin{rem}
As before, the functions used in Proposition \ref{Prop:WWspecfunc} are just two possible examples that could have been chosen. Any known result for Riemann--Liouville or Caputo differintegrals of any particular functions can now easily be extended to an analogous result on weighted differintegrals with respect to functions, with the functions divided by $w(x)$ on left and right sides of the identity.
\end{rem}

\subsection{Examples}

In this subsection, we discuss some particular choices of functions for both $w(x)$ and $\phi(x)$ which will lead to interesting special cases of weighted fractional calculus with respect to functions, some of which are already well known and studied in the literature.

\begin{ex}
If $w(x)=k$ is a constant, then the operators of $w$-weighted fractional calculus with respect to $\phi$, as given in Definition \ref{Def:WW}, are exactly the same as the operators of the original fractional calculus with respect to $\phi$, as given in Definition \ref{Def:WRTF}. Indeed, this is the only case where $w$-weighted fractional calculus with respect to $\phi$ reduces to fractional calculus with respect to a function, because it is the only case when the first-order operator $D_{\phi(x);w(x)}=\frac{1}{\phi'(x)}\cdot\left(\frac{\mathrm{d}}{\mathrm{d}x}+\frac{w'(x)}{w(x)}\right)$ becomes simply a function times $\frac{\mathrm{d}}{\mathrm{d}x}$.

If $\phi(x)=x$, then the operators of $w$-weighted fractional calculus with respect to $\phi$, as given in Definition \ref{Def:WW}, are exactly the same as the operators of $w$-weighted fractional calculus, as given in Definition \ref{Def:wRL&wC}. Indeed, this is the only case where $w$-weighted fractional calculus with respect to $\phi$ reduces to a case of weighted fractional calculus, because it is the only case when the first-order operator $D_{\phi(x);w(x)}=\frac{1}{\phi'(x)}\cdot\left(\frac{\mathrm{d}}{\mathrm{d}x}+\frac{w'(x)}{w(x)}\right)$ contains simply the operator $\frac{\mathrm{d}}{\mathrm{d}x}$ without any function multiplier.
\end{ex}

\begin{ex} \label{Ex:temperedWRTF}
If $w(x)=e^{\beta\phi(x)}$ is an exponential function of $\phi$, then the operators of $w$-weighted fractional calculus with respect to $\phi$, as given in Definition \ref{Def:WW}, are precisely those of tempered fractional calculus with respect to a function, as defined in \cite{fahad-fernandez-rehman-siddiqi}. Therefore, tempered fractional calculus with respect to a function forms an overlap between the general class of weighted fractional operators with respect to functions \cite{jarad-abdeljawad-shah} and the general class of fractional operators with analytic kernels with respect to functions \cite{oumarou-fahad-djida-fernandez}.
\end{ex}

\begin{ex} \label{Ex:Hadamardtype}
If $w(x)=x^{\beta}$ is a power function and $\phi(x)=\log(x)$ is the natural logarithm function, then the operators of $w$-weighted fractional calculus with respect to $\phi$ become precisely those of Hadamard-type fractional calculus, defined \cite{kilbas,butzer-kilbas-trujillo} as follows:
\begin{align*}
	\prescript{H}{a}I^{\alpha,\beta}_xf(x)&=\frac{1}{\Gamma(\alpha)}\int_a^x\left(\frac{t}{x}\right)^{\beta}\left(\log\frac{x}{t}\right)^{\alpha-1}\frac{f(t)}{t}\,\mathrm{d}t,\qquad\beta\in\mathbb{C},\mathrm{Re}(\alpha)>0; \\
	\prescript{HR}{a}D^{\alpha,\beta}_xf(x)&=\left(x\cdot\frac{\mathrm{d}}{\mathrm{d}x}+\beta\right)^n\prescript{H}{a}I^{n-\alpha,\beta}_xf(x),\qquad\beta\in\mathbb{C},\mathrm{Re}(\alpha)\geq0; \\
	\prescript{HC}{a}D^{\alpha,\beta}_xf(x)&=\prescript{H}{a}I^{n-\alpha,\beta}_x\left(x\cdot\frac{\mathrm{d}}{\mathrm{d}x} + \beta\right)^nf(x),\qquad\beta\in\mathbb{C},\mathrm{Re}(\alpha)\geq0.
\end{align*}
This model of fractional calculus was also studied in \cite{fahad-fernandez-rehman-siddiqi}, where it was determined to be a special case of tempered fractional calculus with respect to a function, namely the case with respect to the logarithm function.
\end{ex}

\begin{rem}
Note that choosing $w(x)=e^{\beta\phi(x)}$ in Example \ref{Ex:temperedWRTF}, rather than $w(x)=e^{\beta x}$ as we did to obtain tempered fractional calculus as a special case of weighted fractional calculus in Example \ref{Ex:tempered}, is necessary because we are defining weighted fractional calculus with respect to functions by applying the $M_w$ conjugation after the $Q_{\phi}$ conjugation. As discussed in Remark \ref{Rem:whichway}, this means that our operators of Definition \ref{Def:WW} are
\[
\text{weighted }(\text{fractional calculus with respect to functions})
\]
and not
\[
(\text{weighted fractional calculus})\text{ with respect to functions}.
\]
Thus, choosing $w(x)=e^{\beta x}$ in Definition \ref{Def:wRL&wC} gives tempered fractional calculus, but choosing $w(x)=e^{\beta x}$ in Definition \ref{Def:WW} does not give the with-respect-to-functions version of tempered fractional calculus: instead, it gives the tempered version of fractional calculus with respect to functions. On the other hand, if we swapped the order of the $M_w$ and $Q_{\phi}$ operators in the definition, then Hadamard-type fractional calculus would be given by $w(x)=e^{\beta x}$ and $\phi(x)=\log(x)$, as it is tempered fractional calculus taken with respect to the natural logarithm function.

Explicitly, the operators of Hadamard-type fractional calculus are given by the following conjugation relations:
\begin{align*}
\prescript{H}{a}I^{\alpha,\beta}_xf(x)&=M_{x^{\beta}}^{-1}\circ Q_{\log}\circ\prescript{RL}{\log(a)}I^{\alpha}_x\circ Q_{\log}^{-1}\circ M_{x^{\beta}} \\
&=Q_{\log}\circ M_{e^{\beta x}}^{-1}\circ\prescript{RL}{\log(a)}I^{\alpha}_x\circ M_{e^{\beta x}}\circ Q_{\log}^{-1}; \\
\prescript{HR}{a}D^{\alpha,\beta}_xf(x)&=M_{x^{\beta}}^{-1}\circ Q_{\log}\circ\prescript{RL}{\log(a)}D^{\alpha}_x\circ Q_{\log}^{-1}\circ M_{x^{\beta}} \\
&=Q_{\log}\circ M_{e^{\beta x}}^{-1}\circ\prescript{RL}{\log(a)}D^{\alpha}_x\circ M_{e^{\beta x}}\circ Q_{\log}^{-1}; \\
\prescript{HC}{a}D^{\alpha,\beta}_xf(x)&=M_{x^{\beta}}^{-1}\circ Q_{\log}\circ\prescript{C}{\log(a)}D^{\alpha}_x\circ Q_{\log}^{-1}\circ M_{x^{\beta}} \\
&=Q_{\log}\circ M_{e^{\beta x}}^{-1}\circ\prescript{C}{\log(a)}D^{\alpha}_x\circ M_{e^{\beta x}}\circ Q_{\log}^{-1},
\end{align*}
which was already known from \cite[Theorem 3.2]{fahad-fernandez-rehman-siddiqi}.
\end{rem}

\begin{ex}
If $w(x)=x^{\sigma\eta}$ and $\phi(x)=x^{\sigma}$ are power functions, with $\mathrm{Re}(\eta)>0$ and $\sigma>0$, then the operators of $w$-weighted fractional calculus with respect to $\phi$, as given in Definition \ref{Def:WW}, are almost exactly those of the so-called Erd\'elyi--Kober fractional calculus, which we define following \cite[\S18.1]{samko-kilbas-marichev} and \cite[\S2.6]{kilbas-srivastava-trujillo} as follows. The fractional integral is
\[
\prescript{E}{a}I^{\alpha;\sigma,\eta}_xf(x)=\frac{\sigma x^{-\sigma(\alpha+\eta)}}{\Gamma(\alpha)}\int_a^x\big(x^{\sigma}-t^{\sigma}\big)^{\alpha-1}t^{\sigma\eta+\sigma-1}f(t)\,\mathrm{d}t,\qquad\mathrm{Re}(\alpha)>0,
\]
while the fractional derivative (of Riemann--Liouville type) is
\[
\prescript{ER}{a}D^{\alpha;\sigma,\eta}_xf(x)=x^{-\sigma\eta}\left(\frac{1}{\sigma x^{\sigma-1}}\cdot\frac{\mathrm{d}}{\mathrm{d}x}\right)^nx^{\sigma(\eta+n)}\prescript{E}{a}I^{n-\alpha;\sigma,\eta+\alpha}_xf(x),\qquad\mathrm{Re}(\alpha)\geq0,
\]
and the fractional derivative of Caputo type, defined more recently in \cite{odibat-baleanu}, is:
\[
\prescript{EC}{a}D^{\alpha;\sigma,\eta}_xf(x)=x^{\sigma n}\prescript{E}{a}I^{n-\alpha;\sigma,\eta+\alpha}_xx^{-\sigma(\eta+\alpha)}\left(\frac{1}{\sigma x^{\sigma-1}}\cdot\frac{\mathrm{d}}{\mathrm{d}x}\right)^nx^{\sigma(\eta+\alpha)}f(x),\qquad\mathrm{Re}(\alpha)\geq0,
\]
where in both of the last two cases the natural number $n$ is defined by $n-1<\mathrm{Re}(\alpha)\leq n$, or in other words $n=\lfloor\mathrm{Re}(\alpha)\rfloor+1$.

It is clear that the Erd\'elyi--Kober integral is related to the $w$-weighted fractional integral with respect to $\phi$ as follows:
\[
\prescript{E}{a}I^{\alpha;\sigma,\eta}_xf(x)=x^{-\sigma\alpha}\cdot\prescript{RL}{a}I^{\alpha}_{x^{\sigma};x^{\sigma\eta}}f(x).
\]
The original Erd\'elyi--Kober derivative can also be related to the $w$-weighted fractional derivative with respect to $\phi$ of Riemann--Liouville type, as follows:
\[
\prescript{ER}{a}D^{\alpha;\sigma,\eta}_xf(x)=x^{\sigma\alpha}\cdot\prescript{RL}{a}D^{\alpha}_{x^{\sigma};x^{\sigma(\eta+\alpha)}}f(x).
\]
And the Caputo--type Erd\'elyi--Kober derivative can be related similarly to the $w$-weighted fractional derivative with respect to $\phi$ of Caputo type, as follows:
\[
\prescript{EC}{a}D^{\alpha;\sigma,\eta}_xf(x)=x^{\sigma\alpha}\cdot\prescript{C}{a}D^{\alpha}_{x^{\sigma};x^{\sigma(\eta+\alpha)}}f(x).
\]
These relationships were already noted in \cite[Eq. (2.6.9)]{kilbas-srivastava-trujillo}, using a different notation of $M_{\eta}$ and $N_{\sigma}$ operators to give a relationship between the Erd\'elyi--Kober integral and the Riemann--Liouville integral, and in \cite[Eq. (16)--(24)]{odibat-baleanu}, using a direct formulation involving substitutions of power functions.

It is also interesting to note that, by our Proposition \ref{Prop:WWanalcont}, the Erd\'elyi--Kober fractional derivative is the unique analytic continuation of the Erd\'elyi--Kober fractional integral $\prescript{E}{a}I^{\alpha;\sigma,\eta}_xf(x)$ from the original domain $\mathrm{Re}(\alpha)>0$ to the whole complex plane for $\alpha$, under the convention that
\[
\prescript{E}{a}I^{\alpha;\sigma,\eta}_xf(x)=\prescript{ER}{a}D^{-\alpha;\sigma,\eta+\alpha}_xf(x),\qquad\mathrm{Re}(\alpha)\leq0.
\]
As usual with analytic continuation relations, this fact will be generally useful in proving results about Erd\'elyi--Kober derivatives when the corresponding results for Erd\'elyi--Kober integrals are already known, simply by extension using the fact of analytic continuation.
\end{ex}

\subsection{Laplace transform and convolution}

We shall now study the $w$-weighted Laplace transform with respect to $\phi$, an integral transform which is ideally suited for studying fractional differential equations which are both weighted and with respect to functions. Most of the results of this subsection were already seen in \cite{jarad-abdeljawad-shah}, but we shall now see how to prove them much more quickly and easily by using operational calculus.

\begin{defn}[\cite{jarad-abdeljawad-shah}]
	Let $ f : [a, \infty) \to \mathbb{C} $ be a real-valued or complex-valued function, and let $w$ and $\phi$ be functions as above. Then the $w$-weighted Laplace transform of $f$ with respect to $\phi$ is defined by
	\begin{equation} \label{WWLT}
		\mathcal{L}_{\phi;w(x)} \left\{ f(x) \right\}=F(s)=\int_{a}^{\infty} e^{-s[\phi(x)-\phi(a)]} w(x) f(x)\phi'(x) \,\mathrm{d}x,
	\end{equation}
	for all $s\in\mathbb{C}$ such that this integral converges.
\end{defn}

\begin{thm} \label{Thm:WWLTcomp}
The $w$-weighted Laplace transform with respect to $\phi$ can be written as a composition of the usual Laplace transform with multiplication and composition operators, as follows:
\begin{equation}
\label{WWLT:comp}
\mathcal{L}_{\phi;w(x)}=\mathcal{L}\circ Q_{\phi(x)-\phi(a)}^{-1}\circ M_{w(x)},
\end{equation}
where $M$ and $Q$ are the operators defined above in \eqref{Mdef} and in Lemma \ref{Lem:WRTFconjug} respectively, and where we assume the increasing function $\phi$ satisfies $\phi(x)\to\infty$ as $x\to\infty$.
\end{thm}

\begin{proof}
We check the effect of applying the three operators from the right-hand side, one by one, on an appropriate function $f$:
\begin{align*}
M_{w(x)}f(x)&=w(x)f(x); \\
Q_{\phi(x)-\phi(a)}^{-1}\circ M_{w(x)}f(x)&=w\Big(\phi^{-1}\big(x+\phi(a)\big)\Big)f\Big(\phi^{-1}\big(x+\phi(a)\big)\Big); \\
\mathcal{L}\circ Q_{\phi(x)-\phi(a)}^{-1}\circ M_{w(x)}f(x)&=\int_0^{\infty}e^{-sx}w\Big(\phi^{-1}\big(x+\phi(a)\big)\Big)f\Big(\phi^{-1}\big(x+\phi(a)\big)\Big)\,\mathrm{d}x \\
&=\int_{a}^{\infty}e^{-s[\phi(t)-\phi(a)]}w(t)f(t)\phi'(t)\,\mathrm{d}t,
\end{align*}
where in the last step we substituted $x=\phi(t)-\phi(a)$ and used the assumption on the infinite limiting behaviour of $\phi$.
\end{proof}

\begin{cor}
If $\phi:[0,\infty)\to[0,\infty)$ is an increasing bijection, then the $w$-weighted Laplace transform with respect to $\phi$ can be written as a composition of the usual Laplace transform with multiplication and composition operators, as follows:
\[
\mathcal{L}_{\phi;w(x)}=\mathcal{L}\circ Q_{\phi}^{-1}\circ M_{w(x)},
\]
where $M_{w(x)}$ and $Q_{\phi}$ are the operators defined above in \eqref{Mdef} and  in Lemma \ref{Lem:WRTFconjug} respectively.
\end{cor}

\begin{proof}
This is the case $a=0$, $\phi(0)=0$ of the preceding theorem.
\end{proof}

\begin{cor}
If $a\in\mathbb{R}$ is any real number and $\phi:[a,\infty)\to[\phi(a),\infty)$ is a bijection, then the Laplace transform with respect to $\phi$ can be written in terms of the usual Laplace transform as follows:
\[
\mathcal{L}_{\phi}=\mathcal{L}\circ Q_{\phi(x)-\phi(a)}^{-1}.
\]
\end{cor}

\begin{proof}
This is the case $w(x)=1$ of the preceding theorem. Note that it also provides a generalisation of the results of \cite{fahad-fernandez-rehman-siddiqi}, where it was assumed that $\phi(0)=0$. The operational-calculus approach in \cite{fahad-fernandez-rehman-siddiqi} was therefore used only with fractional differintegrals having lower limit $0$, but now those results can be straightforwardly extended to fractional differintegral operators with general lower limit $a\in\mathbb{R}$.
\end{proof}

As corollaries of the relation \eqref{WWLT:comp}, we obtain the following results.

\begin{cor}
If $w$ and $f$ are such that $w(x)f(x)$ is of $\phi$-exponential order $c$ (as defined in \cite{fahad-fernandez-rehman-siddiqi}), then the $w$-weighted Laplace transform $F(s)$ of $f$ with respect to $\phi$ exists for $\mathrm{Re}(s)>c$.
\end{cor}

\begin{cor}
The inverse $w$-weighted Laplace transform with respect to $\phi$ exists for any function which has a classical inverse Laplace transform, and it may be written as follows:
	\[
		\mathcal{L}_{\phi;w(x)}^{-1}=M_{w(x)}^{-1}\circ Q_{\phi(x)-\phi(a)}\circ\mathcal{L}^{-1},
	\]
	or in other words
	\[
		\mathcal{L}^{-1}_{\phi;w(x)} \left\{F(s)\right\} =\frac{1}{2\pi i w(x)} \int_{c-i\infty}^{c+i\infty} e^{s[\phi(x)-\phi(a)]} F(s) \,\mathrm{d}s.
	\]
\end{cor}

\begin{cor}
	If $f$ is a function which has a classical Laplace transform $F(s)$, then the $w$-weighted Laplace transform with respect to $\phi$ of the function $\left(M_{w(x)}^{-1}\circ Q_{\phi(x)-\phi(a)}f\right)(x)=\frac{f\big(\phi(x)-\phi(a)\big)}{w(x)}$ is also $F(s)$:
	\[
	\mathcal{L}\{f(x)\}=F(s)\quad\Rightarrow\quad\mathcal{L}_{\phi;w(x)}\left\{\frac{f\big(\phi(x)-\phi(a)\big)}{w(x)}\right\}=F(s).
	\]
\end{cor}

The following theorem gives the natural relationship between the weighted Laplace transform with respect to a function and the operators of weighted fractional calculus with respect to a function.

\begin{thm}
	Let $ \alpha >0 $ and let $f$ be a continuous function on $[0,\infty)$ which is of $w$-weighted $\phi$-exponential order, where $w$ is a continuous weight function and $\phi:[a,\infty)\to[\phi(a),\infty)$ is an increasing bijection. Then we have the following results.
	\begin{enumerate}
	\item \[\mathcal{L}_{\phi;w(x)} \left\{ \left( 	\prescript{RL}{a}I^{\alpha}_{\phi(x);w(x)}f\right)(x) \right\}= s^{-\alpha}\mathcal{L}_{\phi;w(x)} \left\{f(x)\right\}.\]
	\item Let $ n-1\leq\mathrm{Re}(\alpha)<n\in\mathbb{Z}^+ $, and assume that $\prescript{RL}{a}D^{\alpha}_{x;w(x)}f$ is continuous on $[a,\infty)$ and of $w$-weighted $\phi$-exponential order. Then
	\[\mathcal{L}_{\phi;w(x)}\left\{\left(\prescript{RL}{a}D^{\alpha}_{\phi(x);w(x)}f\right)(x) \right\} = s^{\alpha}\mathcal{L}_{\phi;w(x)}\{f(x)\} - w(a^+)\sum_{i=0}^{n-1} s^{n-i-1}\left(\prescript{RL}{a}I^{n-i-\alpha}_{\phi(x);w(x)}f\right)(a^+).\]
	\item Let $ n-1\leq\mathrm{Re}(\alpha)<n\in\mathbb{Z}^+ $, and assume that $\prescript{RL}{a}D^{n}_{x;w(x)}f$ is continuous on $[a,\infty)$ and of $w$-weighted $\phi$-exponential order. Then
	\[\mathcal{L}_{\phi;w(x)} \left\{ \left(\prescript{C}{a}D^{\alpha}_{\phi(x);w(x)}f\right)(x) \right\} =s^{\alpha} \mathcal{L}_{\phi;w(x)}\{f(x)\} - w(a^+)\sum_{i=0}^{n-1} s^{\alpha-i-1}\Big(D_{\phi(x);w(x)}\Big)^if(a^+).\]
	\end{enumerate}
\end{thm}

\begin{proof}
All of these results follow from combining the composition results of Theorem \ref{Thm:WWconjug} and Theorem \ref{Thm:WWLTcomp} with the classical facts on Laplace transforms of fractional integrals and derivatives which were quoted previously in the proof of Theorem \ref{inttrans}. However, in this case the $Q$ operators are a little more tricky to deal with, since we have both $Q_{\phi}$ and $Q_{\phi(x)-\phi(a)}$, closely related but not the same operator, involved in the same manipulation.

For the fractional integral, we have:
\begin{align*}
\mathcal{L}_{\phi;w(x)}\circ\prescript{RL}{a}I^{\alpha}_{\phi(x);w(x)}&=\left(\mathcal{L}\circ Q_{\phi(x)-\phi(a)}^{-1}\circ M_{w(x)}\right)\circ\left(M_{w(x)}^{-1}\circ Q_{\phi}\circ\prescript{RL}{\phi(a)}I^{\alpha}_x\circ Q_{\phi}^{-1}\circ M_{w(x)}\right) \\
&=\mathcal{L}\circ\left(Q_{\phi(x)-\phi(a)}^{-1}\circ Q_{\phi}\right)\circ\prescript{RL}{\phi(a)}I^{\alpha}_x\circ Q_{\phi}^{-1}\circ M_{w(x)},
\end{align*}
and the composition $Q_{\phi(x)-\phi(a)}^{-1}\circ Q_{\phi}$ is equivalent to a linear substitution:
\[
\left(Q_{\phi(x)-\phi(a)}^{-1}\circ Q_{\phi}\right)g(x)=g\left(x+\phi(a)\right),
\]
so this operator has the following effect on the fractional integral operator:
\begin{align*}
Q_{\phi(x)-\phi(a)}^{-1}\circ Q_{\phi}\circ\prescript{RL}{\phi(a)}I^{\alpha}_x\;g(x)&=\frac{1}{\Gamma(\alpha)}\int_{\phi(a)}^{x+\phi(a)}\left(x+\phi(a)-t\right)^{\alpha-1}g(t)\,\mathrm{d}t \\
&=\frac{1}{\Gamma(\alpha)}\int_{0}^{x}(x-u)^{\alpha-1}g\left(u+\phi(a)\right)\,\mathrm{d}u \\
&=\prescript{RL}{0}I^{\alpha}_x\circ Q_{\phi(x)-\phi(a)}^{-1}\circ Q_{\phi}\;g(x),
\end{align*}
and therefore 
\begin{align*}
\mathcal{L}_{\phi;w(x)}\circ\prescript{RL}{a}I^{\alpha}_{\phi(x);w(x)}&=\mathcal{L}\circ\prescript{RL}{0}I^{\alpha}_x\circ\left(Q_{\phi(x)-\phi(a)}^{-1}\circ Q_{\phi}\right)\circ Q_{\phi}^{-1}\circ M_{w(x)} \\
&=M_{s^{\alpha}}^{-1}\circ\mathcal{L}\circ Q_{\phi(x)-\phi(a)}^{-1}\circ M_{w(x)} \\
&=M_{s^{\alpha}}^{-1}\circ\mathcal{L}_{\phi;w(x)},
\end{align*}
which is the stated relation for fractional integrals.

For fractional derivatives, since the linear substitution $Q_{\phi(x)-\phi(a)}^{-1}\circ Q_{\phi}$ commutes with the $\frac{\mathrm{d}}{\mathrm{d}x}$ operator, the same manipulations as above give rise to:
\begin{align*}
\mathcal{L}_{\phi;w(x)}\circ\prescript{RL}{a}D^{\alpha}_{\phi(x);w(x)}&=\mathcal{L}\circ\prescript{RL}{0}D^{\alpha}_x\circ Q_{\phi(x)-\phi(a)}^{-1}\circ M_{w(x)}, \\
\mathcal{L}_{\phi;w(x)}\circ\prescript{C}{a}D^{\alpha}_{\phi(x);w(x)}&=\mathcal{L}\circ\prescript{C}{0}D^{\alpha}_x\circ Q_{\phi(x)-\phi(a)}^{-1}\circ M_{w(x)}.
\end{align*}
For the Riemann--Liouville case, we then have
\begin{align*}
\mathcal{L}_{\phi;w(x)} \left\{\prescript{C}{a}D^{\alpha}_{\phi(x);w(x)}f(x) \right\} &=\mathcal{L} \left\{\prescript{C}{a}D^{\alpha}_{\phi(x);w(x)}\Big[Q_{\phi(x)-\phi(a)}^{-1}\circ M_{w(x)}f(x)\Big] \right\} \\
&\hspace{-3cm}=s^{\alpha} \mathcal{L}\{f(x)\} - \sum_{i=0}^{n-1} s^{n-i-1}\lim_{x\to0^+}\left(\prescript{}{0}I^{n-i-\alpha}_{x}\Big[Q_{\phi(x)-\phi(a)}^{-1}\circ M_{w(x)}f(x)\Big]\right) \\
&\hspace{-3cm}=s^{\alpha} \mathcal{L}\{f(x)\} - \sum_{i=0}^{n-1} s^{n-i-1}\lim_{x\to a^+}\Big(Q_{\phi(x)-\phi(a)}\circ\prescript{}{0}I^{n-i-\alpha}_{x}\circ Q_{\phi(x)-\phi(a)}^{-1}\circ M_{w(x)}f(x)\Big) \\
&\hspace{-3cm}=s^{\alpha} \mathcal{L}\{f(x)\} - w(a^+)\sum_{i=0}^{n-1} s^{n-i-1}\left(\prescript{RL}{a}I^{n-i-\alpha}_{\phi(x);w(x)}f\right)(a^+),
\end{align*}
which is the stated relation, and similarly for the Caputo case.
\end{proof}

\begin{defn} [\cite{jarad-abdeljawad-shah}]
The $ w $-weighted $\phi$-convolution of two real-valued or complex-valued functions $f,g:[a,\infty)\to\mathbb{C}$ is the function $f\ast_{w(x)} g$ defined by 
	\begin{equation}\label{WWconvolution}
		\Big(f*_{\phi;w(x)} g\Big)(x)= \frac{1}{w(x)}\int_{a}^{x} w\Big(\phi^{-1}\big(\phi(x)+\phi(a)-\phi(t)\big)\Big)f\Big(\phi^{-1}\big(\phi(x)+\phi(a)-\phi(t)\big)\Big) w(t)g(t)\phi'(t)\,\mathrm{d}t.
	\end{equation}
\end{defn}

\begin{thm} \label{Thm:WWconvol}
The $w$-weighted $\phi$-convolution is related to classical convolution via the following formula:
	\begin{equation}
		f\ast_{\phi;w(x)} g = M_{w(x)}^{-1}\circ Q_{\phi(x)-\phi(a)}\left(Q_{\phi(x)-\phi(a)}^{-1}\circ M_{w(x)}f\right)\ast\left(Q_{\phi(x)-\phi(a)}^{-1}\circ M_{w(x)}g\right),
	\end{equation}
	or equivalently, using the notation of binary operations and conjugation,
	\[
	\ast_{\phi;w(x)}=M_{w(x)}^{-1}\circ Q_{\phi(x)-\phi(a)}\circ\ast\circ\left(Q_{\phi(x)-\phi(a)}^{-1}\circ M_{w(x)},Q_{\phi(x)-\phi(a)}^{-1}\circ M_{w(x)}\right),
	\]
	where the $M$ and $Q$ operators are as in \eqref{Mdef} and Lemma \ref{Lem:WRTFconjug} respectively. 
\end{thm}

\begin{proof}
Clearly, the convolution of $Q_{\phi(x)-\phi(a)}^{-1}\circ M_{w(x)}f$ and $Q_{\phi(x)-\phi(a)}^{-1}\circ M_{w(x)}g$ can be written as
\begin{align*}
\left(Q_{\phi(x)-\phi(a)}^{-1}\circ M_{w(x)}f\right)&\ast\left(Q_{\phi(x)-\phi(a)}^{-1}\circ M_{w(x)}g\right)(x) \\
&\hspace{-2.5cm}=\int_0^xw\Big(\phi^{-1}\big(x-u+\phi(a)\big)\Big)f\Big(\phi^{-1}\big(x-u+\phi(a)\big)\Big) \\ &\hspace{3cm}\times w\Big(\phi^{-1}\big(u+\phi(a)\big)\Big)g\Big(\phi^{-1}\big(u+\phi(a)\big)\Big)\,\mathrm{d}u \\
&\hspace{-2.5cm}=\int_a^{\phi^{-1}(x+\phi(a))}w\Big(\phi^{-1}\big(x+2\phi(a)-\phi(t)\big)\Big)f\Big(\phi^{-1}\big(x+2\phi(a)-\phi(t)\big)\Big)w(t)g(t)\phi'(t)\,\mathrm{d}t,
\end{align*}
where we substituted $u=\phi(t)-\phi(a)$ in the integral. Applying $Q_{\phi(x)-\phi(a)}$ to this expression turns it into
\[
\int_a^{x}w\Big(\phi^{-1}\big(\phi(x)+\phi(a)-\phi(t)\big)\Big)f\Big(\phi^{-1}\big(\phi(x)+\phi(a)-\phi(t)\big)\Big)w(t)g(t)\phi'(t)\,\mathrm{d}t,
\]
and then dividing by $w(x)$ gives exactly the $w$-weighted $\phi$-convolution $\left(f*_{\phi;w(x)} g\right)(x)$, as required.
\end{proof}

The result of Theorem \ref{Thm:WWconvol} for the $w$-weighted $\phi$-convolution, combined with the previous result of Theorem \ref{Thm:WWLTcomp} for the $w$-weighted $\phi$-Laplace transform, enable the following result, already seen in \cite[Theorem 5.9]{jarad-abdeljawad-shah}, to be proved immediately by composition of operators.

\begin{cor}
If $f,g:[0,X]\to\mathbb{C}$ are piecewise continuous and of $ w $-weighted $\phi$-exponential order $c>0 $, then
	\begin{equation*}
		\mathcal{L}_{\phi;w(x)} \left\{ f\ast_{\phi;w(x)} g\right\}=\mathcal{L}_{\phi;w(x)}\{f\}\mathcal{L}_{\phi;w(x)}\{g\}.
	\end{equation*}
\end{cor}

Finally, the $w$-weighted $\phi$-Laplace transform can be used to establish a regularity condition for the solutions to weighted fractional differential equations with respect to functions, for example the following initial value problem:
\begin{align}
	\prescript{C}{0}D^{\alpha}_{\phi(x);w(x)}\boldsymbol{y}(x)&=A\boldsymbol{y}(x)+\boldsymbol{g}(x), \qquad x\geq0, \label{deWW}
	\\ \boldsymbol{y}(0)&= \boldsymbol{\eta},  \label{icWW}
\end{align}
where $0<\alpha<1$ is fixed and $ A = \left(a_{ij}\right) $ is an $n\times n$ constant matrix and $\boldsymbol{g}$ is a continuous $n$-dimensional vector-valued function and $\boldsymbol{\eta}$ is a constant $n$-dimensional vector.

\begin{thm}
	Assume that the system $\eqref{deWW}-\eqref{icWW}$ has a unique continuous solution $ \boldsymbol{y} $. If $ \boldsymbol{g}$ is continuous on $[0,\infty) $ and $ w $-weighted $\phi$-exponentially bounded, then $ \boldsymbol{y} $ and $ \prescript{C}{0}D^{\alpha}_{x;w(x)}\boldsymbol{y} $ are both $ w $-weighted $\phi$-exponentially bounded too.
\end{thm}

\begin{proof}
	Using the result proved in \cite[Theorem 3.1]{Kexue} together with the conjugation relations given by Theorem \ref{Thm:WWconjug}, we obtain the required result.
\end{proof}

\section{Conclusions} \label{Sec:concl}

This paper has constructed a formal mathematical analysis of the structure of the so-called weighted fractional calculus, and also of the same operators taken with respect to an increasing function. Both of these frameworks can be seen as general classes of operators, and within each general class there is an operational conjugation relation which connects every operator in the class back to the basic Riemann--Liouville and Caputo fractional calculi via multiplication and composition. Although these conjugation relations have been mentioned previously in the literature, some previous research on weighted fractional calculus with respect to functions had failed to take them into account. The conjugation relations make the whole theory more efficient by enabling shorter simpler proofs based on knowledge of the RL and Caputo fractional calculi.

The general class of weighted fractional calculus operators includes such well-known types of operators as those of tempered fractional calculus and Kober--Erd\'elyi fractional calculus. The even more general class of weighted fractional calculus operators with respect to functions includes operators such as those of Hadamard-type fractional calculus and Erd\'elyi--Kober fractional calculus, as well as their extensions to be taken with respect to an arbitrary increasing function. Thus, the theory developed in this paper can be useful in the understanding of various types of fractional calculus that have already appeared in the literature and been found with applications worthy of discussion. The whole general class itself, weighted fractional calculus with respect to functions, has been found useful in probability theory and variational calculus.

Corresponding to each fractional differintegral operator in these general classes, there is both a Laplace-type integral transform and a convolution operation. These too can be expressed as operational compositions (not necessarily conjugations) of the classical Laplace transform and convolution together with some multiplication or composition operators. Just like the classical Laplace transform, the weighted Laplace transform (with respect to a function) can be used to solve differential equations with appropriate differential operators, as we have demonstrated briefly in this paper.

The current work is purely theoretical, but much remains to be done in the direction of fractional differential equations and applications. It is presumed that many methods used for solving fractional differential equations, either analytically or numerically, can be extended to corresponding methods which apply to solve weighted fractional differential equations or weighted fractional differential equations with respect to functions. Our current study, particularly the conjugation relations proved herein, will be vital in establishing such extensions.

In the pure mathematical direction, it is possible to extend these general classes of operators still further. The class of fractional differintegrals with analytic kernels and the class of fractional differintegrals with respect to functions have already been combined into a single superclass; the same is done here with the class of weighted fractional differintegrals and the class of fractional differintegrals with respect to functions. Future work may focus on combining the class of weighted differintegrals with the class of differintegrals with analytic kernels, in order to obtain yet another superclass of operators in fractional calculus.

%\section*{Acknowledgements}

\end{document}